\documentclass{amsart}
\usepackage{amsmath,amssymb,epsfig,amscd,amsthm}
\usepackage[bookmarks,colorlinks,breaklinks]{hyperref}  
\hypersetup{linkcolor=blue,citecolor=blue,filecolor=blue,urlcolor=blue}

\usepackage{verbatim}

\numberwithin{equation}{section}

\newtheorem{lemma}[equation]{Lemma}
\newtheorem{thm}[equation]{Theorem}
\newtheorem{conjecture}[equation]{Conjecture}
\newtheorem{cor}[equation]{Corollary}
\newtheorem{prop}[equation]{Proposition}

\newtheorem{example}[equation]{Example}

\newtheorem{defi}[equation]{Definition}

\theoremstyle{remark}

\newtheorem{remark}[equation]{Remark}

\renewcommand{\bar}[1]{#1\llap{$\overline{\phantom{\rm#1}}$}}

\newcommand{\lra}{\longrightarrow}

\DeclareMathOperator{\ord}{{ord}}

\newcommand{\hhat}{\widehat{h}}
\DeclareMathOperator{\supp}{{Supp}}

\newcommand{\N}{{\mathbb N}}
\newcommand{\Z}{{\mathbb Z}}
\newcommand{\Q}{{\mathbb Q}}
\newcommand{\R}{{\mathbb R}}
\newcommand{\C}{{\mathbb C}}

\newcommand{\Kbar}{{\bar{K}}}

\newcommand{\cB}{\mathcal{B}}
\newcommand{\cC}{\mathcal{C}}

\newcommand{\cS}{\mathcal{S}}
\newcommand{\cY}{\mathcal{Y}}
\newcommand{\cZ}{\mathcal{Z}}

\newcommand{\cD}{\mathcal{D}}
\newcommand{\cE}{\mathcal{E}}
\newcommand{\cM}{\mathcal{M}}
\newcommand{\cP}{\mathcal{P}}
\newcommand{\cK}{\mathcal{K}}

\newcommand{\fp}{\mathfrak p}
\newcommand{\fq}{\mathfrak q}

\newcommand{\bfA}{{\mathbf A}}

\newcommand{\fo}{\mathfrak o}

\newcommand{\bP}{{\mathbb P}}
\newcommand{\cO}{\mathcal{O}}

\DeclareMathOperator{\id}{id}

\begin{document}



\title[Squarefree Doubly Primitive Divisors]{Squarefree Doubly Primitive Divisors in Dynamical Sequences}

\author{D.~Ghioca}
\address{
Dragos Ghioca\\
Department of Mathematics\\
University of British Columbia\\
Vancouver, BC V6T 1Z2\\
Canada
}
\email{dghioca@math.ubc.ca}

\author{K.~D.~Nguyen}
\address{
Khoa D.~Nguyen
Department of Mathematics\\
University of British Columbia\\
Vancouver, BC V6T 1Z2\\
Canada
}
\email{dknguyen@math.ubc.ca}

\author{T.~J.~Tucker}
\address{
Thomas Tucker\\
Department of Mathematics\\ 
University of Rochester\\
Rochester, NY 14627\\
USA
}
\email{thomas.tucker@rochester.edu}

\keywords {write keywords}
\subjclass[2010]{Primary 11G50; Secondary 14G99}

\thanks{The first author was partially supported by an NSERC Discovery Grant. The second author thanks the Pacific Institute of Mathematical Sciences for its generous support. The third
  author was partially supported by  an NSF grant.}


 \begin{abstract}
Let $K$ be a number field or a function field of characteristic $0$, let $\varphi\in K(z)$ with $\deg(\varphi)\geq 2$, and let $\alpha\in \bP^1(K)$. Let $S$ be a finite set
of places of $K$ containing all the archimedean ones and 
the primes where $\varphi$ has bad reduction. 
After excluding all the natural counter-examples, 
we define a subset $\bfA(\varphi,\alpha)$ of
$\N_0\times\N$  and show that for all but finitely many
$(m,n)\in \bfA(\varphi,\alpha)$ there is a prime
$\fp\notin S$ such that 
$\ord_{\fp}(\varphi^{m+n}(\alpha)-\varphi^m(\alpha))=1$
and $\alpha$ has portrait $(m,n)$ under the action of
$\varphi$ modulo $\fp$. 
This latter condition implies 
$\ord_{\fp}(\varphi^{u+v}(\alpha)-\varphi^u(\alpha))\leq 0$
for $(u,v)\in \N_0\times\N$ satisfying
$u<m$ or $v<n$.
Our proof assumes a conjecture of Vojta
for $\bP^1\times\bP^1$ in the number field case and is unconditional in the function field case thanks to a deep theorem of Yamanoi. This paper extends earlier work of
Ingram-Silverman, Faber-Granville, and the authors.	
\end{abstract}

\maketitle

\section{Introduction}
\label{intro}

Throughout this paper, $\N$ denotes the set of positive integers and $\N_0=\N\cup\{0\}$.
We start by recalling the definition from \cite{portrait} for a \emph{preperiodicity portrait}, or simply \emph{portrait} for a rational function $\varphi$ defined over a field $F$ and a point $\alpha\in \bP^1(F)$. We say that $(m,n)\in \N_0\times \N$ is the portrait of $\alpha$ (under $\varphi$) if  $\varphi^{m}(\alpha)$ is periodic of minimum period $n$ under $\varphi$, while $m$ is the smallest nonnegative integer such that $\varphi^{m}(\alpha)$ is periodic (as always in dynamics, we denote by $\varphi^k=\varphi\circ \cdots \circ \varphi$ composed with itself $k$ times). We call $m$ the  \emph{preperiod} of
$\alpha$ and call $n$ the \emph{minimum} or \emph{exact period} of $\alpha$.

From now on, let $K$ be either a number field or a function field of transcendence
degree $1$ over an algebraically closed field $\kappa$ of characteristic $0$. By a place
$\fp$
of $K$, we mean an equivalence class of absolute values on $K$ that are trivial on the constant field $\kappa$ if $K$ is a function field.  Let $M_K$ denote
the set of places of $K$. 
Each nonarchimedean place $\fp$ gives rise to a valuation ring $\fo_\fp$
and a maximal ideal denoted (by an abuse of notation) by $\fp$; we let $M_K^0$ be the set of all nonarchimedean places of $K$ (also called \emph{finite} places). We denote by $k_\fp$ the residue field of $K$ at the place $\fp\in M_K^0$; if $K$ is a function field, then $k_\fp$ is canonically isomorphic to $\kappa$.  
For $\fp\in M_K^0$, let
$\ord_\fp$ denote the additive  
valuation on $K$ normalized so that $\ord_\fp(K)=\Z\cup\{\infty\}$. 
For every $\fp\in M_K^0$, we
have a well-defined reduction map $r_\fp:\bP^1(K)\lra \bP^1(k_\fp)$. Given a non-constant function $\varphi(x)\in K(x)$, for all but finitely many $\fp\in M_K^0$, by ``reducing the coefficients of $\varphi$ modulo $\fp$'', the morphism
$\varphi:\ \bP^1_K\rightarrow \bP^1_K$
induces a morphism $\bar{\varphi}:\ \bP^1_{k_\fp}\rightarrow\bP^1_{k_\fp}$
of the same degree
such that $\bar{\varphi}(r_\fp(a))=r_\fp( \varphi(a))$
for every $a\in\bP^1(K)$
 and we say that $\varphi$ has good reduction modulo $\fp$ (see \cite[Chapter~2]{Silverman-book}). 
For $\fp\in M_K^0$, if $K$ is a number field, let
$N_{\fp}=\log(\# k_\fp)$ ; otherwise, let $N_{\fp}=1$.
For every finite subset $S$ of $M_K$, let 
$N_S:=\displaystyle\sum_{\fp\in S\cap M_K^0} N_{\fp}$.
Let $h_K$ denote the Weil height (over $K$)
on $\bP^1(\bar{K})$ and for every $\varphi(x)\in K(x)$
having degree at least $2$, we can
define the canonical height
$\hhat_{\varphi,K}$ on $\bP^1(\bar{K})$. The readers
are referred to Section~\ref{sec:absv and heights}
or \cite[Chapter~3]{Silverman-book} for more details.

Given a sequence $(a_n)_{n\geq 0}$ of elements of $K$, we may ask if for every (sufficiently large) $n$ there exists a prime $\fp$ such that $\ord_\fp(a_n)>0$
while $\ord_\fp(a_m)\leq 0$ for every $m<n$. Such primes
are called primitive divisors of the sequence $(a_n)$.
This highly interesting question has a long history starting from
work of Bang \cite{Bang}, Zsigmondy \cite{Zsigmondy},
and Schinzel \cite{Schinzel} in the context of
the multiplicative group to further work
in the setting of elliptic curves (for examples, see \cite{EMW} and \cite{Ingram2007}).
First results in the context of arithmetic dynamics where
$a_n=\varphi^n(\alpha)$ for a given $\varphi(x)\in K(x)$
and $\alpha\in \bP^1(K)$ are obtained by Ingram and Silverman \cite{IS}. After \cite{IS}, there are various papers on primitive divisors in dynamical sequences
\cite{FG}, \cite{DH2012}, \cite{Krieger}, and especially
\cite{GNT2013} in which the existence of primitive
divisors is established for the function field case and  conditionally on ABC for the number field case. Analogous
results for the arithmetic dynamics of higher dimensional
varieties are obtained by Silverman \cite{Silverman2013}
assuming Vojta's conjecture for projective spaces.

A harder question asked by Ingram-Silverman
is the existence of a prime $\fp$ such that
$\ord_{\fp}(\varphi^{m+n}(\alpha)-\varphi^m(\alpha))>0$
while $\ord_{\fp}(\varphi^{u+v}(\alpha)-\varphi^v(\alpha))\leq 0$ if $u<m$ or $v<n$. Such primes
are called doubly primitive divisors (for the double sequence $a_{m,n}=\varphi^{m+n}(\alpha)-\varphi^n(\alpha)$) by Faber and Granville \cite{FG} who discovered
certain counter-examples to the question by Ingram-Silverman and proposed a modified question. In \cite{portrait}, we answer a variant of the Ingram-Silverman-Faber-Granville question for function fields
(and the proof in \cite{portrait} can be adapted to settle the number field case assuming ABC). More
explicit results for the special case of unicritical 
polynomials are obtained in recent work of Doyle
\cite{Doyle2016}, \cite{DoylePreprint}.

In number theory, whenever one has $\ord_{\fp}(a)>0$, it
is natural to ask whether $\ord_{\fp}(a)=1$ (i.e. whether $\fp$ is a squarefree factor of $a$). In fact, the
existence of \emph{squarefree} primitive divisors
in the sequence $(a_n:=\varphi^n(\alpha))_{n\geq 0}$
is also proved in \cite{GNT2013} and has an interesting
application to the structure of certain iterated Galois groups \cite[Section~6]{GNT2013}. The aim of this paper
is to study the existence of the so-called 
squarefree doubly primitive divisors:

\begin{defi}\label{def:squarefree doubly}
Let $\varphi(x)\in K(x)$ with $\deg(\varphi)\geq 2$
and let $\alpha\in\bP^1(K)$ which is not $\varphi$-periodic. Let $\fp$ be a prime of good reduction.
We say that $\alpha$ has portrait $(m,n)$ modulo $\fp$
if $r_{\fp}(\alpha)$ has portrait $(m,n)$
under the reduction $\bar{\varphi}$ of $\varphi$.
If, in addition, $\infty\notin\{\varphi^{m+n}(\alpha),\varphi^m(\alpha)\}$ and $\ord_{\fp}(\varphi^{m+n}(\alpha)-\varphi^m(\alpha))=1$ then we say that
$\alpha$ has \emph{squarefree} portrait $(m,n)$ modulo $\fp$.
\end{defi}
If $\alpha$ has portrait $(m,n)$ modulo $\fp$
then $\ord_{\fp}(\varphi^{u+v}(\alpha)-\varphi^u(\alpha))\leq 0$ when $u<m$ or $v<n$; this explains the connection to the concept of doubly primitive divisors in Faber-Granville \cite{FG}. Given $(m,n)$, the existence of
$\fp$ such that $\alpha$ has squarefree portrait $(m,n)$
modulo $\fp$ is also of interest in complex dynamics. Perhaps the simplest example is the well-known result of Gleason that when $K=\C(t)$, $\varphi(x)=x^2+t$, and
$\alpha=0$, for every $n\in\N$, the Gleason polynomial 
$\varphi^{n}(\alpha)-\alpha\in \C[t]$
has only simple roots. To illustrate the kind of results proved in this paper without introducing several technical definitions, we state the following theorem:

\begin{thm}\label{thm:main}
	Let $K$ be a number field or a function field. 
	In the number field case, assume Vojta's conjecture
	for $\bP^1\times\bP^1$ (see Conjecture~\ref{conj:surface}). 
	Let $\varphi(x)\in K(x)$ with 
	$d:=\deg(\varphi)\geq 2$, let 
	$S$ be a finite set of places of 
	$K$
	containing the archimedean ones and all primes of
	bad reduction of $\varphi$, and let 
	$\tau$ be a positive number.
	Then there is a constant $C_1=C_1(K,\varphi,N_S,\tau)$
	depending only on $K$, $\varphi$, $N_S$, and
	$\tau$ such that the following holds. 
	For every $\alpha\in \bP^1(K)$
	such that $\hhat_{\varphi,K}(\alpha)\geq \tau$,
	for all positive integers $m>C_1$
	and $n>C_1$,
	there is a prime $\fp\in M_K\setminus S$ such that
	$\alpha$ has squarefree portrait $(m,n)$ modulo 
	$\fp$.
\end{thm}

\begin{remark}
We emphasize the fact that $C_1$ depends on a lower
bound $\tau$ of $\hhat_{\varphi,K}(\alpha)$ instead 
of $\alpha$ itself. This means our results (Theorem~\ref{thm:main} and Theorem~\ref{thm:extra}) are essentially uniform in $\alpha$. Indeed, if $K$ is a number field,
Northcott's principle gives a positive
lower bound on the canonical height of
wandering points in $\bP^1(K)$. This also holds in 
the function field case as long as $\varphi$ is not isotrivial thanks to a result of Baker \cite{Baker2009}. 
\end{remark}

While Theorem~\ref{thm:main} is effective in the function field case, its effectiveness in the number field case depends on the effectiveness of Vojta's conjecture
for $\bP^1\times\bP^1$ (Conjecture~\ref{conj:surface}). Nevertheless, even if Theorem~\ref{thm:main} is effective, the 
resulting $C_1$ should be large and the theorem does not say anything about small values of $m$ or $n$. For example, fix $m=2017$, we cannot conclude from Theorem~\ref{thm:main} that for all sufficiently large $n$, there is $\fp$ such that $\alpha$ has squarefree portrait $(2017,n)$. \emph{The ultimate goal of this paper} is to
identify a subset $\bfA(\varphi,\alpha)$
of $\N_0\times\N$ by excluding all the natural counter-examples and prove that for all but finitely many $(m,n)\in \bfA(\varphi,\alpha)$, there is $\fp\in M_K^0\setminus S$ such that $\alpha$ has squarefree portrait $(m,n)$ modulo $\fp$. First, we need the following definition taken from \cite{GNT2013}:

\begin{defi}\label{def:dynamically unramified}
Let $\varphi(x)\in K(x)$ and $a\in\bP^1(\bar{K})$. We say
that $\varphi$ is dynamically unramified over $a$ if
there are infinitely many $b\in\bP^1(\bar{K})$
such that for some $n\in\N$, $\varphi^n(b)=a$
and $\varphi^n$ is unramified at $b$.
In other words, $a$ has an infinite
backward orbit consisting of unramified points
of $\varphi$.
\end{defi}

We are now able to define the set of ``admissible'' $(m,n)$ which could possibly become a squarefree portrait
after reducing modulo some prime $\fp$:
\begin{defi}\label{def:A1A2A}
Let $\varphi\in K(x)$ with $\deg(\alpha)\geq 2$ and 
let $\alpha\in \bP^1(K)$ that is not $\varphi$-preperiodic.
\begin{itemize}
\item [(i)] Let $A_1(\varphi,\alpha)$ be the set of
$m\in\N_0$ satisfying the following condition. If $m=0$, we
require that $\varphi$ is dynamically unramified over $\alpha$; if $m\geq 1$, we require that there exists
$\eta\in\bP^1(\bar{K})\setminus\{\varphi^{m-1}(\alpha)\}$ such that $\varphi(\eta)=\varphi^m(\alpha)$, $\varphi$ is unramified at $\eta$, and dynamically unramified over $\eta$.

\item [(ii)] Let $A_2(\varphi)$ be the set
of $n\in \N$ satisfying the following condition. There is
a $\varphi$-periodic point
$\beta\in\bar{K}$ with exact period $n$ such that
$x-\beta$ is a squarefree factor of
$\varphi^n(x)-x$ and there is $\eta\in\bP^1(\bar{K})\setminus \{\varphi^{n-1}(\beta)\}$ such that
$\varphi(\eta)=\beta$, $\varphi$ is unramified at $\eta$,
and dynamically unramified over $\eta$.

\item [(iii)] Define $\bfA(\varphi,\alpha):=A_1(\varphi,\alpha)\times A_2(\varphi)$. 
\end{itemize}
\end{defi}

Note that this definition makes sense over any field (i.e. not necessarily a number field or function field).
When $m\geq 1$, the condition in Definition~\ref{def:A1A2A}(i) says that $\varphi^m(\alpha)$ has
an infinite backward orbit that starts from $\eta\neq\varphi^{m-1}(\alpha)$ and consists of unramified points.
The condition on $\eta$ in Definition~\ref{def:A1A2A}(ii)
says that $\beta$ has an infinite backward orbit that starts from $\eta\neq \varphi^{n-1}(\beta)$
and consists of unramified points. Roughly speaking,
the ``bad'' set $\N_0\setminus A_1(\varphi,\alpha)$
(respectively $\N\setminus A_2(\varphi)$)
is essentially the set of $m$ (respectively $n$) that
either belongs to the ``bad'' set $Y(\varphi,\alpha)$
(respectively $X(\varphi)$) defined in \cite[Definition~1.2]{portrait} or fails a certain
dynamical unramifiedness condition. Although Definition~\ref{def:A1A2A} looks slightly complicated, we briefly
explain, through counter-examples, why $\bfA(\varphi,\alpha)$ is the largest set (up to adding \emph{finitely many} elements of $\N_0\times\N$)
where we can hope for the existence of squarefree
doubly primitive divisors.

\begin{example}\label{eg:A1 is largest}
Consider $\varphi(x)=(x-\alpha)(x-\delta)^2$ (with $\delta\neq \alpha$) so that $1\notin A_1(\varphi,\alpha)$. The only pre-image of
$\varphi(\alpha)=0$ that is not $\alpha$ is $\delta$
over which $\varphi$ is ramified. Now $\varphi^{n+1}(x)=(\varphi^{n}(x)-\alpha)(\varphi^n(x)-\delta)^2$.
Therefore every squarefree divisor of 
$\varphi^{n+1}(\alpha)-\varphi(\alpha)=\varphi^{n+1}(\alpha)$
must be a divisor of $\varphi^n(\alpha)-\alpha$.
Hence $(1,n)$ is not a squarefree portrait after reducing any prime $\fp$.
\end{example}

\begin{example}\label{eg:A2 is largest}
For simplicity, assume $\varphi$ is a polynomial, and 
let $n\notin A_2(\varphi)$.
\begin{itemize}
\item [(a)] If $\varphi^n(x)-x$ does not have a 
squarefree factor (for instance, $n=1$ and $\varphi(x)=x+x^2$), then obviously $\varphi^{m+n}(\alpha)-\varphi^m(\alpha)$ does not have a squarefree factor.

\item [(b)] Now assume $\varphi^n(x)-x$ has a squarefree
factor, but every such factor is of the form $x-\beta'$
where the exact period of $\beta'$ is strictly smaller than $n$. Then every squarefree prime factor
$\fp$
of $\varphi^{m+n}(\alpha)-\varphi^m(\alpha)$
must be a factor of $\varphi^m(\alpha)-\beta'$, and hence
$r_{\fp}(\varphi^m(\alpha))=r_{\fp}(\beta')$
has exact period less than $n$.  Therefore $(m,n)$
cannot be a squarefree portrait.

\item [(c)] Let $\beta_1,\ldots,\beta_k$
be all the points of exact period $n$ such that
$x-\beta_i$ is a squarefree factor of $\varphi^n(x)-x$.
Assume that for each $\beta_i$, we fail
to have an $\eta$ as described in Definition~\ref{def:A1A2A}(ii).  So there is
$M\geq 1$ such that for every $i\in\{1,\ldots,k\}$, every squarefree factor
of $\varphi^M(x)-\beta_i$, if it exists, must have the form
$x-\delta$ where $\varphi^{M-1}(\delta)=\varphi^{n-1}(\beta_i)$. Let $m\geq M$, assume
that $\fp$
is a squarefree prime factor of $\varphi^{m+n}(\alpha)-\varphi^m(\alpha)$ and $r_\fp(\varphi^m(\alpha))$
has exact period $n$.
 By using the factorization of $\varphi^{M+n}(x)-\varphi^M(x)$ induced from the factorization of
$\varphi^n(x)-x$, we have that $\fp$ is a factor
of some $\varphi^{m-M}(\alpha)-\delta$
with $\varphi^{M-1}(\delta)=\varphi^{n-1}(\beta_i)$
for some $i\in\{1,\ldots,k\}$ as mentioned above. However, this implies $r_\fp(\varphi^{m-1}(\alpha))=r_\fp(\varphi^{M-1}(\delta))=r_{\fp}(\beta_i)$
which is also periodic. Hence $(m,n)$ cannot be the portrait
of $\alpha$ modulo $\fp$.
\end{itemize}
\end{example}

Having explained why the set $\bfA(\varphi,\alpha)$
is essentially best possible, we now state the main
result of this paper:
\begin{thm}\label{thm:extra}
	Let $K$, $\varphi$, $d$, $S$,
	and $\tau$ be as in Theorem~\ref{thm:main}. 
	In the number field case, assume Vojta's conjecture
	for $\bP^1\times\bP^1$ (see Conjecture~\ref{conj:surface}). 
	Then there is a finite subset $\Delta=\Delta(K,\varphi,N_S,\tau)$ of $\N_0\times\N$
	such that for every
	$\alpha\in\bP^1(K)$ with
	$\hhat_{\varphi,K}(\alpha)\geq\tau$,
	the following holds.
	Write $\bfA=\bfA(\varphi,\alpha)$;
    for every $(m,n)\in\bfA\setminus\Delta$ satisfying
    $\infty\notin\{\varphi^{m+n}(\alpha),\varphi^n(\alpha)\}$,
	there is a prime $\fp\in M_K\setminus S$ such that
	$\alpha$ has squarefree portrait $(m,n)$ modulo 
	$\fp$.
\end{thm}

We strongly refer the readers to Section~\ref{sec:eg}
where explicit examples are given
to illustrate Theorem~\ref{thm:extra}.
We can also prove that
the sets $A_1(\varphi,\alpha)$ and $A_2(\varphi)$
are co-finite (i.e.~having a finite complement)
in $\N_0$, see Proposition~\ref{prop:cofinite A1A2}.
In fact, for the examples
in Section~\ref{sec:eg}, the set $\bfA(\varphi,\alpha)$
is ``usually'' the whole $\N_0\times\N$.   
Therefore, Theorem~\ref{thm:extra} subsumes
Theorem~\ref{thm:main}. However, since Theorem~\ref{thm:main} is needed to prove Theorem~\ref{thm:extra}
and its simple statement does not involve the definition of $\bfA(\varphi,\alpha)$, we believe presenting it as a separate theorem will benefit the readers.

The organization of this paper is as follows. 
After giving examples in Section~\ref{sec:eg}
and basic results on absolute values and heights in
Section~\ref{sec:absv and heights}, we introduce Vojta's conjecture for $\bP^1\times\bP^1$
in Section~\ref{sec:Vojta}. Assuming this conjecture,
we prove Corollary~\ref{cor:double sequence} which
is the key ingredient for the proof of our main theorems
in the number field case. In Section~\ref{sec:Yamanoi},
Corollary~\ref{cor:double sequence ff} which is the
function field counterpart of Corollary~\ref{cor:double sequence} is proved thanks to a deep theorem of
Yamanoi. The proof of Theorem~\ref{thm:main}
is given in Section~\ref{sec:proof main}
and we finish the paper with the proof of Theorem~\ref{thm:extra} in Section~\ref{sec:proof extra}.

\medskip

{\bf Acknowledgments.} We thank Paul Vojta for useful suggestions.

\section{Examples}\label{sec:eg}
Let $\varphi(x)\in K(x)$ with $d:=\deg(\varphi)\geq 2$.
We will use the simple observation that if 
$a\in\bP^1(K)$ is not contained in the orbit of 
any critical point then $\varphi$ is unramified
at $a$ as well as any point in the backward orbit of $a$.

\subsection{Number field case}
Let $K$ be a number field and $\varphi(x)=x^2+1$. We have $A_2(\varphi)=\N$ thanks to the following:
\begin{itemize}
	\item For every $n\in\N$, the polynomial $P_n(x):=\varphi^n(x)-x$
	has only simple factors. In fact, every root $r$ of $P_n$ is an algebraic integer and this implies:
	$$P_n'(r)=\varphi'(r)\varphi'(\varphi(r))\cdots \varphi'(\varphi^{n-1}(r))-1=2^nr\varphi(r)\cdots \varphi^{n-1}(r)-1\neq 0.$$
	
	\item For every $n\in\N$, every root $r$
	of $P_n(x)$ whose period is strictly smaller than $n$ must be a root of $P_{n/p}(x)$ for a prime divisor $p$ of $n$. From 
	$$\sum_{\text{prime }p\mid n} 2^{n/p}<2^n,$$
	there exists $\beta\in\bar{\Q}$ having exact period $n$.
	
	\item Since the critical point $0$ is not preperiodic, the existence of $\eta$ as in Definition~\ref{def:A1A2A}(ii) is guaranteed.
\end{itemize}

For $\alpha\in K$ that is not $\varphi$-preperiodic,  
there are three (mutually exclusive) cases:
\begin{itemize}
	\item [(1)] There is $M\in\N_0$ such that
	$\varphi^M(\alpha)=1$. Then $A_1(\varphi,\alpha)=
	\N_0\setminus\{M\}$.
	\item [(2)] There is $M\in \N_0$ such that
	$\varphi^M(\alpha)=-1$. Then $\varphi^{M+1}(\alpha)=2$, the only pre-image of $2$ that is not $-1$ is $1$, and $\varphi$ is not dynamically unramified over $1$. Hence $A_1(\varphi,\alpha)=\N_0\setminus\{M+1\}$.
	\item [(3)] $\{1,-1\}\cap\{\varphi^m(\alpha):\ m\geq 0\}=\emptyset$. Then $A_1(\varphi,\alpha)=\N_0$.
\end{itemize}

Theorem~\ref{thm:extra} shows that there is
a finite subset $\Delta$ of $\N_0\times\N$ depending
only on $\varphi$ and $K$ satisfying the following.
For every $\alpha\in K$ that is not $\varphi$-preperiodic, for every $(m,n)\in \bfA(\varphi,\alpha)\setminus \Delta$, there exists $\fp\in M_K^0$ such that
$\alpha$ has squarefree portrait $(m,n)$ modulo $\fp$.

\subsection{Function field case}
Let $K$ be a finite extension of $\C(t)$ and
 $\varphi(x)=x^2+t$. By similar arguments to the previous
 example, we have that $A_2(\varphi)=\N$. To prove that
 $P_n'(r)\neq 0$ for every root $r$ of $P_n(x)$, use
 the fact that if $\fp$ is the place of $\C(t)$ 
 corresponding to the point at infinity of $\bP^1(\C)$
 then (after extending $\ord_\fp$ to $\overline{\C(t)}$)
 we have
 $\ord_{\fp}(2^nr\varphi(r)\cdots\varphi^{n-1}(r))<0$.
 
 Since $\varphi$ is not isotrivial, by a result of
 Baker \cite{Baker2009}, there is a positive
 lower bound $\tau$ (depending only on $\varphi$ and $K$) 
 on $\hhat_{\varphi,K}(\alpha)$
 for every $\alpha\in \bP^1(K)$
 that is not $\varphi$-preperiodic.

 For $\alpha\in K$ that is not $\varphi$-preperiodic,  
there are three (mutually exclusive) cases:
\begin{itemize}
	\item [(1)] There is $M\in\N_0$ such that
	$\varphi^M(\alpha)=t$. Then $A_1(\varphi,\alpha)=
	\N_0\setminus\{M\}$.
	\item [(2)] There is $M\in \N_0$ such that
	$\varphi^M(\alpha)=-t$. Then $\varphi^{M+1}(\alpha)=t^2+t$, the only pre-image of $t^2+t$ that is not $-t$ is $t$, and $\varphi$ is not dynamically unramified over $t$. Hence $A_1(\varphi,\alpha)=
	\N_0\setminus\{M+1\}$.
	\item [(3)] $\{t,-t\}\cap\{\varphi^m(\alpha):\ m\geq 0\}=\emptyset$. Then $A_1(\varphi,\alpha)=\N_0$.
\end{itemize}	
 	
Theorem~\ref{thm:extra} shows that there is
a finite subset $\Delta$ of $\N_0\times\N$ depending
only on $\varphi$ and $K$ satisfying the following.
For every $\alpha\in K$ that is not $\varphi$-preperiodic, for every $(m,n)\in\bfA(\varphi,\alpha)\setminus \Delta$, there exists a prime of good reduction $\fp\in M_K^0$ such that
$\alpha$ has squarefree portrait $(m,n)$ modulo $\fp$.
 	
\section{Absolute values and heights}
\label{sec:absv and heights}

\subsection{Absolute values}

When $K$ is a function field, $M_K=M_K^0$.
When $K$ is a number field, the set
$M_K\setminus M_K^0$ is finite and corresponds
to the collection of real embeddings $K\rightarrow \R$
and pairs of complex conjugate embeddings
$K\rightarrow \C$.
For
every place $v\in M_K$, define:
$$
\Vert x\Vert_v=
\begin{cases}
	e^{N_\fp\ord_\fp(x)} & \mbox{if $v\in M_K^0$ and it corresponds to the prime ideal $\fp$}\\
	\vert\sigma(x)\vert & \mbox{if $v\in M_K^{\infty}$ is real and it corresponds to 
	    $\sigma:\ K\rightarrow \R$}\\
	\vert\sigma(x)\vert^2 & \mbox{if $v\in M_K^{\infty}$ is complex and it corresponds to $\sigma:\ K\rightarrow \C$}
\end{cases}
$$

We also define $\vert x\vert_v=\Vert x\Vert_v^{1/2}$ if $v$ is complex, and 
$\vert x\vert_v=\Vert x\Vert_v$ otherwise. This way, $\vert\cdot\vert_v$ becomes
an absolute value on $K$ (note that $\Vert\cdot\Vert_v$
does not satisfy the triangle inequality when $v$ is complex).

\subsection{Heights}\label{subsec:heights}
For every real number $y$, define $\log^+(y)=\log\max\{1,y\}$.
We define the Weil height $h_K$ on $\bP^1(\bar{K})$ as follows:
$$h_K(x)=\frac{1}{[K(x):K]}\sum_{\fq\in M_{K(x)}} \log^+\Vert x\Vert_\fq,$$
for every $x\in\bar{K}$, and $h_K(\infty)=0$.
If $L/K$ is a finite extension then $h_L=[L:K]h_K$.
When $K$ is a number field, Northcott's principle states that there are only finitely many elements of $\bP^1(\bar{K})$
whose height and degree are bounded above by a given constant. For the more general Weil height associated to a Cartier divisor on a projective variety, we refer
the readers to \cite{bg06}.

Let $x$ and $y$ be distinct elements of $\bP^1(K)$.
Let $\cE:=\{\fp\in M_K^0:\ r_\fp(x)=r_\fp(y)\}$, 
we have the following
inequality:
\begin{equation}\label{eq:new Liouville}
\sum_{\fp\in \cE}N_{\fp}\leq \min\{h_K(x),h_K(y)\}+h_K(x)+h_K(y)+c_K,
\end{equation}
where $c_K=0$ in the function field case and 
$c_K=[K:\Q]\log 2$ in the number field case.
To prove this, we assume that $x,y\in K$ since the case $x=\infty$
or $y=\infty$ is easy. Without loss of generality, assume
$h_K(x)\leq h_K(y)$. For every
$\fp\in\cE$, either 
$\vert x\vert_{\fp}>1$
or $\vert x-y\vert_{\fp}<1$. Therefore 
$\sum_{\fp\in\cE}N_{\fp}\leq h_K(x)+h_K(x-y).$
Then \eqref{eq:new Liouville} follows
from the well-known inequality
$h_K(x-y)\leq h_K(x)+h_K(y)+c_K$
\cite[Proposition~1.5.15]{bg06}.

If $\varphi(z)\in\Kbar(z)$ is a rational function of degree  $d\ge 2$, then for each point $x\in\bP^1(\Kbar)$, following \cite{Call-Silverman} we define the canonical height (over $K$) of $x$ under the action of $\varphi$ by:
$$\hhat_{\varphi,K}(x)=\lim_{n\to\infty}\frac{h_K(\varphi^n(x))}{d^n}.$$

According to \cite{Call-Silverman}, there is a constant $C_{\varphi,K}$ depending
only on $K$ and $\varphi$ such that $| h_K(x) - \hhat_\varphi(x)| <
C_{\varphi,K}$ for all $x \in \bP^1(\Kbar)$.  
When $K$ is a number field, Northcott's principle implies
that $\hhat_{\varphi,K}(x)=0$ if and only if $x$ is $\varphi$-preperiodic. Moreover, there is a positive
lower bound depending only on 
$K$ and $\varphi$ for $\hhat_{\varphi,K}(x)$
for every $x\in\bP^1(K)$ that is not $\varphi$-preperiodic. Baker \cite{Baker2009} proves that an analogous result holds in the 
function field case if $\varphi$ is not isotrivial.

\section{A conjecture of Vojta and its consequences}
\label{sec:Vojta}

Throughout this section, let $K$ be a number field. An $M_K$-constant is a collection of real numbers
$(c_v)_{v\in M_K}$ such that $c_v=0$ for all but finitely many $v$.
Let $X$ be a smooth projective variety and $D$ a (Weil or Cartier) divisor
defined over $K$. For each $v\in M_K$, we can define Weil functions:
$$\lambda_{X,D,v}: (X\setminus \supp D)(\C_v)\rightarrow \R$$
satisfying certain functorial properties (see 
\cite[Chapter~8]{Vojta2011} for more details). 
If $D=\displaystyle \sum_{i=1}^k m_i[a]$ is a divisor on $\bP^1_K$ where $k\in\N$ and $m_i\in \Z$ for every $i$, we \emph{always} use the definition:
    $$\lambda_{\bP^1_K,D,\fp}(x)=\displaystyle\sum_{i=1}^k m_i\log^+\frac{1}{\Vert x-a_i\Vert_{\fp}}$$
for $x\notin\{a_1,\ldots,a_k\}$. Note
that when $a_i=\infty$, the formula 
$\displaystyle\log^+\frac{1}{\Vert x-a_i\Vert_{\fp}}$ is interpreted as $\log^+\Vert x \Vert_v$.
 
Define the truncated counting function (see \cite[Chapter~22]{Vojta2011}):
$$N^{(1)}(X,D,x):=\sum_{\fp\in M_K^0}\min\{\lambda_{X,D,\fp}(x),N_\fp\}$$
for $x\in X(K)$ not lying in the support of $D$. 

From now on, we work over the ambient variety $X=\bP^1\times\bP^1$. The following is a special case of Vojta's conjecture (see
\cite[Conjecture~22.5(b)]{Vojta2011}):
\begin{conjecture}\label{conj:surface}
Let $K$ be a number field, let $D$ be a normal crossing divisor on $X=(\bP^1_K)^2$, let $\cK=\cO(-2,-2)$ be the canonical sheaf on $X$. Then for any $\epsilon>0$
there is a proper Zariski closed subset $Z$ of $X$, depending 
on $K$, $D$, and $\epsilon$, such that for all $C\in \R$, the inequality
$$N^{(1)}(X,D,(a,b))\geq h_{\cK(D),K}(a,b)-\epsilon(h(a)+h(b))-C$$
holds for all but finitely many $(a,b)\in (X\setminus Z)(K)$.
\end{conjecture}

\begin{remark}
Note that if $D\cong\cO(q,r)$ then $\cK(D)\cong\cO(q-2,r-2)$ and we can choose the definition:
$$h_{\cK(D),K}(a,b)=(q-2)h_K(a)+(r-2)h_K(b)$$
for $(a,b)\in (\bP^1)^2(K)$.
\end{remark}

\begin{lemma}\label{lem:S(K,f)}
Let $K$ be a number field, let $f(x)\in K(x)\setminus K$
be a non-constant rational function. Let $\cZ$
and $\cP$ respectively denote the effective divisors of zeros and poles of $f$ (hence $(f)=\cZ-\cP$). There is a finite
set of places $S(K,f)\subseteq M_K$ depending only on $K$ and
$f$ such that the following holds. For every prime
$\fp\in M_K^0\setminus S(K,f)$, for every $a\in P^1(K)$
that is not a zero or pole of $f$, we have:
\begin{itemize}
\item [(a)] If 
$\lambda_{\bP^1_K,\cP,\fp}(a)>0$
then $\log \Vert f(a)\Vert_{\fp} = \lambda_{\bP^1_K,\cP,\fp}(a)$,
\item [(b)] if $\lambda_{\bP^1_K,\cP,\fp}(a)=0$
then $\log \Vert f(a)\Vert_{\fp} = -\lambda_{\bP^1_K,\cZ,\fp}(a)$.
\end{itemize}
\end{lemma}
\begin{proof}
There is a finite set of places $S_1(K,f)$ such that for every 
$\fp\in M_K^0\setminus S(K,f)$ and for every $a\in \bP^1(K)$
that is not a zero or pole of $f$, the following hold:
\begin{itemize}
	\item [(i)] $\lambda_{\bP^1_K,\cZ,\fp}(a)=\lambda_{\bP^1_K,(f),\fp}(a)+\lambda_{\bP^1_K,\cP,\fp}(a)=-\log\Vert f(a)\Vert_{\fp}+\lambda_{\bP^1_K,\cP,\fp}(a)$.
	\item [(ii)] If $\lambda_{\bP^1_K,\cZ,\fp}(a)>0$
	then $a$ and some zero of $f$ have the same reduction modulo $\fp$.
	\item [(iii)] If $\lambda_{\bP^1_K,\cP,\fp}(a)>0$
	then $a$ and some pole of $f$ have the same reduction modulo $\fp$. 
\end{itemize}
Extend $S_1(K,f)$ to $S(K,f)$ such that for every $\fp\in M_K^0\setminus S(K,f)$, any pole of $f$ and any zero of $f$ have
different reduction modulo $\fp$.
\end{proof}

Let $f(x)\in K(x)$, $a\in \bP^1(K)$
 that is not a pole of $f$, and $b\in K$ such that $f(a)\neq b$, let $\cZ(f,a,b):=\cZ_K(f,a,b)$ be the set of primes $\fp$ of $K$
 such that $\ord_{\fp}(f(a)-b)\geq 1$ or, equivalently, 
 $\Vert f(a)-b\Vert_{\fp}<1$. We have:
\begin{prop}\label{prop:uniform}
Assume Conjecture~\ref{conj:surface} holds. Let $K$ be a number field, let $f(x)\in K(x)$ be 
a rational function of degree $d>0$.
For every $\epsilon>0$,
there exist a proper Zariski closed subset $Z$ of $X:=(\bP^1_K)^2$ and 
a constant $C_1:=C_1(K,f,\epsilon)$ such that
\begin{equation}\label{eq:prop uniform}
\sum_{\fp\in \cZ(f,a,b)} N_{\fp}
\geq (d-2-\epsilon)h_K(a)-(2+\epsilon)h_K(b)-C_1
\end{equation}
for every $(a,b)\in (X\setminus Z)(K)$ such that
$a$ is not a pole of $f$ and $b\in K$ satisfying $f(a)\neq b$.
\end{prop}
\begin{proof}

Let $D$ be the effective divisor defined by the equation $y=f(x)$
in $X$. Let $v$ be a place of $K$, we
now define $\lambda_{X,D,v}$. 
Let $\cP$ denote the effective divisor of $\bP^1_K$
corresponding to the poles (counted with multiplicity) of $f$,
then write $D_{\cP}:=\cP\times\bP^1_K$ to denote the pull-back to $X$. Let $D_{\infty}$ be the divisor
$\bP^1_K\times \{\infty\}$. Write $R=(f(x)-y)$ to denote the 
principal divisor generated by the rational function $f(x)-y$.
From 
$R=D-D_{\cP}-D_{\infty}$
we can take
$$\lambda_{X,D,v}:=\lambda_{X,R,v}+\lambda_{X,D_{\cP},v}+\lambda_{X,D_{\infty},v}.$$
We define
$\lambda_{X,R,v}$, $\lambda_{X,D_{\cP},v}$, and $\lambda_{X,D_{\infty},v}$
as follows. For every $(a,b)\in X(K)$ such that $a$ is not a pole of $f$, 
$b\neq \infty$, and $f(a)\neq b$, we have:
$$\lambda_{X,R,v}(a,b)=-\log \Vert f(a)-b\Vert_v,$$
$$\lambda_{X,D_{\cP},v}(a,b)=\lambda_{\bP^1_K,\cP,v}(a)$$
$$\lambda_{X,D_{\infty},v}(a,b)=\lambda_{\bP^1_K,[\infty],v}(b)=\log^+\Vert b\Vert_v.$$
Therefore, we can define:
\begin{equation}\label{eq:combine 3 lambda}
\lambda_{X,D,v}(a,b)=-\log \Vert f(a)-b\Vert_v+\lambda_{\bP^1_K,\cP,v}(a)+\log^+\Vert b\Vert_v.
\end{equation}

By Conjecture~\ref{conj:surface}, given $\epsilon>0$, there exist a proper Zariski closed subset
$Z$ of $X$ and a constant $C_2$ depending on $K$, $f$, and $\epsilon$
such that: 
\begin{equation}\label{eq:apply Vojta}
N^{(1)}(X,D,(a,b))\geq (d-2)h_K(a)-h_K(b)-\epsilon(h_K(a)+h_K(b)) 
-C_2
\end{equation}
for $(a,b)\in (X\setminus Z)(K)$. If $a$ is not a pole
of $f$, $b\neq \infty$, and $f(a)\neq b$, from \eqref{eq:combine 3 lambda} we have:
\begin{align}\label{eq:def of N1}
N^{(1)}(X,D,(a,b))=&\sum_{\fp\in M_K^0}\min\{-\log\Vert f(a)-b\Vert_\fp+\lambda_{\bP^1_K,\cP,\fp}(a)+\log^+\Vert b\Vert_{\fp},N_{\fp}\}.
\end{align}

By choosing $C_1>\vert d-2-\epsilon\vert \cdot \max\{h_K(z):\ f(z)=0\}$, inequality \eqref{eq:prop uniform}
holds when $a$ is zero of $f$. 
From now on, we assume the extra condition that $a$ is not a zero of $f$. Let $S(K,f)$ be the set of places of $K$ as in the conclusion
of Lemma~\ref{lem:S(K,f)}. We partition $M_K^0$
into three sets $M_K^0(1)$, $M_K^0(2)$, and $M_K^0(3)$
(depending on $K$, $f$, $a$, and $b$)
as follows.

Firstly, let $M_K^0(1):=\{\fp\in M_K^0:\ 
\log^+\Vert b\Vert_p>0\ 
\text{or}\ \fp\in S(K,f)\}$. We then have:
\begin{align}\label{eq:firstly}
\begin{split}	
	\sum_{\fp\in M_K^0(1)}\min\{-\log\Vert f(a)-b\Vert_{\fp}+\lambda_{\bP^1_K,\cP,\fp}(a)
+\log^+\Vert b\Vert_{\fp}, N_{\fp}\}&\leq \sum_{\fp\in M_K^0(1)}N_{\fp}\\
&\leq h_K(b)+C_3
\end{split}
\end{align}
where $C_3:=\displaystyle\sum_{\fp\in S(K,f)}N_{\fp}$.

Secondly, let $M_K^0(2):=
\{\fp\in M_K^0:\ \log^+\Vert b\Vert_{\fp}=0,\ 
\fp\notin S(K,f),\  \text{and}\ \lambda_{\bP^1_K,\cP,\fp}(a)>0\}$. For every $\fp\in M_K^0(2)$, Lemma~\ref{lem:S(K,f)}
gives
$\log\Vert f(a)\Vert_{\fp}=\lambda_{\bP^1_K,\cP,\fp}(a)>0$. Since
$\Vert b\Vert_{\fp}\leq 1$, we have $\log\Vert f(a)-b\Vert_{\fp}=\log\Vert f(a)\Vert_{\fp}=\lambda_{\bP^1_K,\cP,\fp}(a)$. Therefore: 

\begin{equation}\label{eq:secondly}
\sum_{\fp\in M_K^0(2)}\min\{-\log\Vert f(a)-b\Vert_{\fp}+\lambda_{\bP^1_K,\cP,\fp}(a)
+\log^+\Vert b\Vert_{\fp}, N_{\fp}\}=0.
\end{equation}

Thirdly, let $\cM_K^0(3):=\{\fp\in M_K^0:\ \log^+\Vert b\Vert_{\fp}=0,\ 
\fp\notin S(K,f),\ \text{and}\ \lambda_{\bP^1_K,\cP,\fp}(a)=0\}$.
For every $\fp\in\cM_K^0(3)$, we have:
$$\min\{-\log\Vert f(a)-b\Vert_{\fp}+\lambda_{\bP^1_K,\cP,\fp}(a)
+\log^+\Vert b\Vert_{\fp}, N_{\fp}\}=\min\{-\log\Vert f(a)-b\Vert_{\fp},N_{\fp}\}.$$
Moreover, let $\cZ$ denote the divisor of zeros of $f$, Lemma~\ref{lem:S(K,f)} gives
$0\leq \lambda_{\bP^1,\cZ,\fp}(a)=-\log\Vert f(a)\Vert_{\fp}$,
hence $\Vert f(a)\Vert_{\fp}\leq 1$. Since
$\Vert b\Vert_{\fp}\leq 1$, we have $\Vert f(a)-b\Vert_{\fp}\leq 1$. Therefore:
\begin{align}\label{eq:thirdly}
\begin{split}
	&\sum_{\fp\in M_K^0(3)}\min\{-\log\Vert f(a)-b\Vert_p+\lambda_{\bP^1_K,\cP,\fp}(a)
+\log^+\Vert b\Vert_{\fp},N_{\fp}\}\\ 
  &\leq \sum_{\fp\in M_K^0(3)}\min\{\ord_{\fp}(f(a)-b),1\}N_{\fp}\leq \sum_{\fp\in \cZ(f,a,b)} N_{\fp}. 
\end{split}
\end{align}

Combining \eqref{eq:apply Vojta}, \eqref{eq:def of N1}, \eqref{eq:firstly}, \eqref{eq:secondly},
and \eqref{eq:thirdly}, we have:
\begin{equation}
\sum_{\fp\in \cZ(f,a,b)}N_{\fp} + h_K(b)+C_3
\geq (d-2-\epsilon)h_K(a)-(1+\epsilon)h_K(b)-C_2
\end{equation}
for $(a,b)\in (X\setminus Z)(K)$ such that
$a$ is not a zero or pole of $f$, $b\neq \infty$, and $f(a)\neq b$. This finishes the proof.
\end{proof}

Proposition~\ref{prop:uniform} cannot be applied if $\bP^1\times\{b\}$
is a (horizontal) component of the exceptional variety 
$Z$. We have the following complement result:
\begin{prop}\label{prop:Roth-abc}
Assume Conjecture~\ref{conj:surface} holds. Let $K$, $f$,
and $\epsilon$ be as in Proposition~\ref{prop:uniform}.
Let $b\in K$ and let $d_b$ be the number of zeroes
of $f(x)-b$ counted without multiplicities(i.e.~the cardinality of $f^{-1}(b)$). Then there is a 
constant $C_4:=C_4(K,f,\epsilon,b)$
such that
\begin{equation}\label{eq:prop Roth-abc}
\sum_{\fp\in \cZ(f,a,b)} N_{\fp}
\geq (d_b-2-\epsilon)h_K(a)-C_4
\end{equation}
for every $a\in \bP^1(K)$ with $f(a)\notin\{\infty,b\}$.
\end{prop}
\begin{proof}
Let $\delta$ be the divisor of zeroes counted without multiplicities of $f(x)-b$
in $\bP^1_K$. In other words, over $\bar{K}$, the divisor
$\delta$ is reduced and consists of the $d_b$ points
in the set $f^{-1}(b)$. Let $D=\delta\times\bP^1_K$
be a divisor of $X:=(\bP^1_K)^2$, we have
$D\cong \cO(d_b,0)$.

There is a finite set of places $S'\subset M_K^0$
depending only on $K$ and $f$ such that 
for every $\fp\in M_K^0\setminus S'$ and
for every $a\in \bP^1(K)$ with $f(a)\notin\{\infty,b\}$,
we have
$\vert f(a)-b\vert_\fp<1$ if and only if 
$\vert a-z_0\vert_\fp<1$ for some $z_0\in f^{-1}(b)$
(as always, if $z_0=\infty$ we interpret 
$\vert a-z_0\vert_\fp$
as $\vert 1/a\vert_\fp$). Hence, for every such $a$
and for every $\beta\in\bP^1(K)$, we have:
\begin{equation}\label{eq:Roth-abc 1}
\sum_{\fp\in \cZ(f,a,b)} N_{\fp}+\sum_{\fp\in S'}N_\fp\geq N^{(1)}(\bP^1_K,\delta,a)=N^{(1)}(X,D,(a,\beta)).  
\end{equation}

By Conjecture~\ref{conj:surface}, there
is a proper Zariski closed subset $Z$ of $X$
depending on $K$, $f$, $\epsilon$, and $b$
such that the inequality
\begin{equation}\label{eq:Roth-abc use Vojta}
N^{(1)}(X,D,(a,\beta))\geq (d_b-2-\epsilon)h_K(a)-\epsilon h_K(\beta)
\end{equation}
holds for all but finitely many 
$(a,\beta)\in (X\setminus Z)(K)$.
We now \emph{fix} a $\beta_0$ such
that $\bP^1_K\times \{\beta_0\}$ is
not a horizontal component of $Z$. Hence there are only
finitely many $a\in \bP^1(K)$ such that
$(a,\beta_0)\in Z(K)$. This together with
\eqref{eq:Roth-abc 1} and \eqref{eq:Roth-abc use Vojta}
finish the proof.
\end{proof}

The goal of this section is the following application to arithmetic dynamics:
\begin{cor}\label{cor:double sequence}
	Assume Conjecture~\ref{conj:surface} holds. Let $K$ be a number field, let $\varphi\in K(z)$ be a rational function of degree $d\geq 2$, let
	$\ell$ be a positive integer. Write $\hhat=\hhat_{\varphi,K}$. For every $\epsilon>0$,
	there exist constants $C_{5}$ and $N\geq \ell$ depending only on $K$, 
	$\varphi$,
	$\ell$,
	and $\epsilon$ such that the following holds. 
	For every $a\in \bP^1(K)$ that is not $\varphi$-preperiodic, $b\in K$, and
	$n\in\N_0$ satisfying , $\hhat(b)\leq \hhat(a)$, 
	$n> N$, and
	$\varphi^n(a)\neq \infty$,
	 let $\cD$ denote the set of primes
	$\fp\in M_K^0$ such that $\ord_{\fp}(\varphi^{n}(a)-b)\geq 1$, we have:
	\begin{equation}
	\sum_{\fp\in \cD}N_{\fp}
	\geq (d_{\ell,b}-\epsilon-2)h_K(\varphi^{n-\ell}(a))
	-(2+\epsilon)h_K(b)-C_{5}
	\end{equation}
	where $d_{\ell,b}$ is the number of zeroes counted
	without multiplicities of 
	$\varphi^{\ell}(z)-b$.
\end{cor}

\begin{proof}  
	We apply Proposition~\ref{prop:uniform} for $f=\varphi^{\ell}$
	and get a proper Zariski closed subset $Z$ of
	$X:=(\bP^1_K)^2$ and a constant $C_6$ such
	that:
	\begin{equation}\label{eq:apply prop uniform}
	\sum_{\fp\in \cZ(f,\alpha,\beta)} N_{\fp}
\geq (d^{\ell}-2-\epsilon)h_K(\alpha)-(2+\epsilon)h_K(\beta)-C_6
\end{equation}
for every $(\alpha,\beta)\in (X\setminus Z)(K)$ such that
$\alpha$ is not a pole of $f$ and $\beta\in K$ satisfying $f(\alpha)\neq \beta$.	
	If some irreducible components of $Z$ are points, we simply increase $C_{6}$
	to take care of them. Hence we may assume that every irreducible
	component of $Z$ is a curve. Let $\{\beta_1,\ldots,\beta_s\}$
	be the (possibly empty) set of
	$\beta\in\bP^1(K)$ such that $\bP^1\times\{\beta\}$
	is a component of $Z$. Let $n\in\N_0$ with $n> \ell$ and 
	$\varphi^n(a)\neq \infty$. 
	From the condition $\hhat(b)\leq \hhat(a)$,
	we have that $\varphi^{n}(a)\neq b$.
	There are two cases.
	
	First, if $b=\beta_i$ for some $1\leq i\leq s$, then Proposition~\ref{prop:Roth-abc}
	(for $f=\varphi^{\ell}$ and $b=\beta_i$)
	gives the desired inequality. It remains to
	consider the case $b\notin\{\beta_1,\ldots,\beta_s\}$. Therefore when $n$ is sufficiently large,
	the point $(\varphi^{n-\ell}(a),b)$ does not belong to
	a vertical or horizontal component of 
	$Z$. Once we can prove that $(\varphi^{n-\ell}(a),b)$ does not belong to a ``non-trivial'' (i.e.~neither vertical nor horizontal) component of $Z$ then \eqref{eq:apply prop uniform}
	with $\alpha=\varphi^{n-\ell}(a)$ and
	$\beta=b$ and the fact that
	$d^{\ell}\geq d_{\ell,b}$
	yield the desired inequality.

		Note that there
		exist positive constants $C_{7}$ and $C_{8}$ (depending on $K$ and $Z$)
		such that for every point $(s,t)\in (\bP^1)^2(K)$, if $(s,t)$ lies in a non-trivial component of $Z$ 
		then:
		\begin{equation}\label{eq:C7 and C8}
		C_{7} h_K(t)+C_{8}>h_K(s).
		\end{equation}
		Since $\vert \hhat-h_K\vert =O(1)$ on $\bP^1(\bar{K})$
		where the constants in $O(1)$ depend only on $K$
		and $\varphi$, inequality
		\eqref{eq:C7 and C8} implies that
		there are positive constants $C_9$ and $C_{10}$
		such that if 
		$(\varphi^{n-\ell}(a),b)$
		is in a non-trivial component of $Z$ then:
		\begin{equation}\label{eq:C9 and C10}
		C_9\hhat(b)+C_{10}>\hhat(\varphi^{n-\ell}(a)).
		\end{equation}	
		When $d^{n-\ell}>C_9$, the inequalities
		$\hhat(a)\geq\hhat(b)$ and \eqref{eq:C9 and C10}
		imply:
		\begin{equation}\label{eq:upper bound}
		\hhat(\alpha)<\frac{C_{10}}{d^{n-\ell}-C_9}.		
		\end{equation}
		Since $K$
		is a number field and $\alpha$ is $\varphi$-preperiodic, 
		there is a positive constant $C_{11}$
		depending only on $K$ and $\varphi$
		such that $\hhat(a)\geq C_{11}$. Hence
		when $n$ is sufficiently large so that
		$C_{11}>\displaystyle\frac{C_{10}}{d^{n-\ell}-C_9}$,
		the point
		$(\varphi^{n-\ell}(a),b)$
		cannot belong to a non-trivial component of $Z$.
		This finishes the proof.
\end{proof}

\section{A theorem of Yamanoi and its consequences}
\label{sec:Yamanoi}

The goal of this section is to prove \emph{unconditionally} 
a variant of Corollary~\ref{cor:double sequence}
for the function field case. Throughout this section,
let $K$ be a function field over the ground field $\kappa$; we recall that this means $\kappa$ is an algebraically closed
field of characteristic 0 and $K$ is the function field of
a curve over $\kappa$. 

When $B$ is a curve over $\kappa$ with function field
$\kappa(B)$, elements
of $\bP^1(\kappa(B))=\kappa(B)\cup\{\infty\}$
are viewed as functions on $B$ by regarding
$\infty$ as the constant function mapping $B$
to $\infty\in\bP^1(\kappa)$. 
The main technical ingredient of
this section is the celebrated theorem of Yamanoi \cite[Theorem~2]{Yamanoi-2004} reformulated as follows:

\begin{thm}[Yamanoi]\label{thm:Yamanoi}
Let $q\in \N$. For all $\epsilon>0$, there exists a positive constant
$C_{12}(q,\epsilon)$ with the following property. Let
$Y$ and $B$ be smooth projective curves over $\kappa$
with a non-constant $\kappa$-morphism $\pi:\ Y\rightarrow B$.
Let $R\in \kappa(Y)$ be a rational function on $Y$. Let
$r_1,\ldots,r_q\in \bP^1(\kappa(B))$ be distinct 
functions on $B$. Assume that 
$R\neq r_i\circ \pi$ for every $i$. Then we have:
\begin{align*}
		(q-2-\epsilon)\deg R&\leq \sum_{1\leq i\leq q} \bar{n}(r_i\circ\pi,R,Y)+2g(Y)\\
							&+C_{12}(q,\epsilon)(\deg\pi)(\max_{1\leq i\leq q}(\deg r_i)+g(B)+1)
		\end{align*}
	
	Here $g(Y)$ (resp. $g(B)$) denotes the genus of $Y$ (resp. $B$), and $\bar{n}(r_i\circ \pi,R,Y)$ is the cardinality of $\{z\in Y: R(z)=r_i\circ \pi(z)\}$.
\end{thm}

 For the rest of this section, fix a smooth projective curve 
 $\cC$ over $\kappa$ satisfying $K=\kappa(\cC)$. 
 For $r\in\bP^1(\bar{K})$, if $r\neq \infty$
 (respectively $r=\infty$)
 express
 $r$ in homogeneous coordinate $[r_0:1]$ (respectively
 $[1:0]$)
 with $r_0\in \bar{K}$ 
 and let $K(r)$ be the field $K(r_0)$ (respectively
 $K$). 
 
\begin{lemma}\label{lem:bounding ramification}
Let $f(x)\in K(x)\setminus K$.  There are constants $C_{13}$
and $C_{14}$
depending on $K$ and $f$ such that the following holds.
For every $b\in \bP^1(K)$ and $r\in \bP^1(\bar{K})$ satisfying $f(r)=b$, 
the $\kappa$-morphism of smooth projective curves
$\cB\rightarrow \cC$ which corresponds to the extension $K(r)/K$
is ramified over at most $C_{13}h_K(b)+C_{14}$ points of $\cC$.
\end{lemma}
\begin{proof}
The conclusion is trivial when $r=\infty$, we may assume
$r\in\bar{K}$. Let $P(x)$ be the minimal polynomial of $r$ over $K$. By 
the inequality $\deg(P)\leq \deg(f)$
and
 $\deg(f)h_K(r)=h_K(b)+O(1)$ where the constants in 
 $O(1)$ depend only on $K$ and $f$, 
there exist
$C_{13}$ and $C_{14}$ depending on $K$
and $f$
such that 
for every place $\fp$ of $\cC$
that lies outside a set of at most
$C_{13}h_K(b)+C_{14}$ places of $\cC$, the polynomial $P(x)$ has $\fp$-adic integral coefficients and its discriminant is a $\fp$-adic unit; in this case we 
know that $\cB/\cC$ is unramified over such $\fp$.
\end{proof}

As in the previous section, for $f(x)\in K(x)$, $a\in \bP^1(K)$ that is not a pole of $f$, and $b\in K$ such that $f(a)\neq b$, let
$\cZ(f,a,b)$ be the set of primes $\fp$
of $K$ such that $\ord_{\fp}(f(a)-b)\geq 1$. We have
the following unconditional counterpart of Proposition~\ref{prop:uniform}: 
\begin{prop}\label{prop:uniform ff}
Let $K$ be a function field of the curve $\cC$ over $\kappa$ as above. Let $f(x)\in K(x)$ be a rational function of degree $d>0$. 
For every
$\epsilon>0$, there exist positive constants $C_{15}$ and $C_{16}$ depending on $K$, $\epsilon$, and $f$ such that 
the following holds. For every $a\in \bP^1(K)$
that is not a pole of $f$ and $b\in K$ with $f(a)\neq b$, we have:
\begin{equation}\label{eq:uniform ff}
\sum_{\fp\in \cZ(f,a,b)}N_{\fp}\geq (d_{f,b}-2-\epsilon)h_K(a)-C_{15}h_K(b)-C_{16}
\end{equation}
where $d_{f,b}$ is the number (counted without multiplicities)
of the solutions of $f(x)=b$ in $\bP^1(\bar{K})$.
\end{prop}

\begin{remark}
In down-to-earth terms, the quantity $\displaystyle\sum_{\fp\in \cZ(f,a,b)} N_{\fp}$
is the number of zeros (counted \emph{without} multiplicities) of the function $f(a)-b$ on the curve $\cC$. The quantities
$h_K(a)$ and $h_K(b)$ are, respectively, the degree of the functions $a$ and $b$.
\end{remark}

\begin{proof}[Proof of Proposition~\ref{prop:uniform ff}]
We may assume $3\leq d_{f,b}$.
Note that $d_{f,b}\leq d$  (equality occurs when $f$ is unramified over $b$), hence we may restrict to 
$b\in K\setminus\{f(a)\}$ with a
\emph{fixed} $q:=d_{f,b}\in\{3,\ldots,d\}$. 
We have distinct elements $r_1,\ldots,r_q\in \bP^1(\bar{K})$ such that $f(r_i)=b$ for every $i$.
Let $\cY=\cB$ be the smooth projective curve over $\kappa$
with function field $L:=K(r_1,\ldots,r_q)$ and
let $\pi=\id:\ \cY\rightarrow \cB$. By Theorem~\ref{thm:Yamanoi}, there exists a constant $C_{17}=C_{17}(q,\epsilon)$
 such that:
\begin{equation}\label{eq:use Yamanoi 1}		
		(q-2-\epsilon)h_L(a)\leq \sum_{1\leq i\leq q} \bar{n}(r_i,a,\cY)+2g(\cY)+C_{17}\left(\max_{1\leq i\leq q}h_L(r_i)+g(\cB)+1\right)
		\end{equation}
 
By Lemma~\ref{lem:bounding ramification}, the
Riemann-Hurwitz theorem for the extension $L/K$, and the
equation $\deg(f)h_K(r_i)=h_K(b)+O(1)$, there exist constants $C_{18}$ and $C_{19}$
depending only on $K$ and $f$ such that:
\begin{equation}\label{eq:genus bound}
	\frac{1}{[L:K]}\left(2g(\cY)+C_{17}\left(\max_{1\leq i\leq q}h_L(r_i)+g(\cB)+1\right) \right)\leq C_{18}h_K(b)+C_{19}.
\end{equation}

Recall that for every finite extension $F$ of $K$ and
for every place $\fp$ of $F$, the notation $r_{\fp}$
denotes the reduction map from $\bP^1(F)$ to $\bP^1(\kappa)$.
By identifying points of $Y$ to places of $L$, for $1\leq i\leq q$, the set $\{z\in Y:\ a(z)=r_i(z)\}$ is exactly
the set of places $\fq$ of $L$ satisfying
$r_{\fq}(a)=r_{\fq}(r_i)$; this latter set of places of $L$ is denoted by 
$S_i$. Let $S_{i,K}$ be the set of places of $K$ lying below
$S_i$.
Let $\cS$ be the set of
places $\fp$ of $K$ satisfying one of the following conditions:
\begin{itemize}
	\item [(i)] $f$ has bad reduction over $\fp$.
	\item [(ii)] $r_{\fp}(r_i)=r_{\fp}(r_j)$ for
	some $i\neq j$.
	\item [(iii)] $\ord_{\fp}(b)<0$.
\end{itemize}
There exist $C_{20}$ and $C_{21}$ depending only on $K$  and $f$ such that 
\begin{equation}\label{eq:|S|}
\vert\cS\vert\leq C_{20}h_K(b)+C_{21}.
\end{equation}
By condition (ii), we have that the sets 
$S_{i,K}\cap (M_K\setminus \cS)$ for $1\leq i\leq q$ are pairwise disjoint. If $\fp\in S_{i,K}\cap (M_K\setminus\cS)$,
then we have $r_{\fp}(f(a))=r_{\fp}(f(r_i))=r_{\fp}(b)$, and hence $\ord_{\fp}(f(a)-b)\geq 1$ since $\ord_{\fp}(b)\geq 0$.
Therefore, we have:
\begin{equation}\label{eq:sum n_i 1}
\frac{1}{[L:K]}\sum_{i=1}^q \bar{n}(r_i,a,\cY)
=\frac{1}{[L:K]} \sum_{i=1}^q\vert S_i\vert
\leq \sum_{i=1}^q \vert S_{i,k}\vert
\leq \vert\cS\vert+\sum_{\fp\in \cZ(f,a,b)}N_{\fp}.
\end{equation}
which, together with \eqref{eq:|S|}, imply
\begin{equation}\label{eq:sum n_i 2}
\frac{1}{[L:K]}\sum_{i=1}^q \bar{n}(r_i,a,\cY)
\leq C_{20}h_K(b)+C_{21}+\sum_{\fp\in \cZ(f,a,b)}N_{\fp}.
\end{equation}
We now divide both sides of \eqref{eq:use Yamanoi 1}
by $[L:K]$ and apply inequalities
\eqref{eq:genus bound} and \eqref{eq:sum n_i 2}
to obtain the desired inequality.
\end{proof}

We
have the following counterpart of 
 Corollary~\ref{cor:double sequence}:
\begin{cor}\label{cor:double sequence ff}
Let $K$, $\cC$, and $\kappa$ be as in Proposition~\ref{prop:uniform ff}. Let $\varphi(x)\in K(x)$ be a 
rational function of degree $d>0$, let $\ell$
be a positive integer. For every $\epsilon>0$, there exist
positive constants $C_{22}$ and $C_{23}$ depending only
on $K$, $\varphi$, $\ell$, and $\epsilon$ such that 
the following hold. For every $a\in \bP^1(K)$, $b\in K$, and $n\in \N_0$ satisfying 
$n\geq \ell$ and 
$\varphi^n(a)\notin\{b,\infty\}$
the set $\cD:=\{\fp\in M_K^0:\ \ord_{\fp}(\varphi^{n}(a)-b)\geq 1\}$
satisfies:
$$\vert\cD\vert=\sum_{\fp\in\cD}N_{\fp}\geq  
(d_{\ell,b}-2-\epsilon)h_K(\varphi^{n-\ell}(a))-C_{22}h_K(b)-C_{23}$$
where $d_{\ell,b}$ is the number (counted without multiplicities) of the solutions of
$\varphi^{\ell}(x)=b$
in $\bP^1(\bar{K})$.
\end{cor}
\begin{proof}
This follows immediately from Proposition~\ref{prop:uniform ff} for $f=\varphi^{\ell}$
and the point $(\varphi^{n-\ell}(a),b)$.
\end{proof}

\section{Proof of Theorem~\ref{thm:main}}
\label{sec:proof main}

Throughout this section, let $K$ be a number field or
a function field over the ground field $\kappa$. We have:
\begin{lemma}\label{lem:easy}
Let $f(x)\in K(x)$ having degree $d>1$ and let $\fp\in M_K^0$
be a prime of good reduction of $f$. Let $\gamma_1,\gamma_2\in \bP^1(K)$
such that $r_{\fp}(f(\gamma_1))\neq r_{\fp}(f(\gamma_2))$ but $r_{\fp}(\gamma_1)\neq r_{\fp}(\gamma_2)$. If $\gamma_1$
is $f$-periodic modulo $\fp$ then $\gamma_2$ is not
$f$-periodic modulo $\fp$.
\end{lemma}
\begin{proof}
This is proved in \cite[pp.~176]{portrait}.
\end{proof}

\begin{lemma}\label{lem:Sb}
Let $f(x)\in K(x)$ having degree $d>1$. There exist constants $C_{24}$ and $C_{25}$
depending only on $K$ and $f$ such that the following hold.
For every $b\in K$ that is not a critical value of $f$, write $f^{-1}(b)=\{r_1,\ldots,r_d\}\subset \bP^1(\bar{K})$,
let $L_b:=K(r_1,\ldots,r_d)$.
Let $\cS_b$ be the subset of $M_K^0$
consisting of primes $\fp$ satisfying one of the following two conditions:
\begin{itemize}
	\item [(i)] $f$ has bad reduction modulo $\fp$.
	\item [(ii)] $f$ has good reduction modulo $\fp$
	and for some prime $\fq$ of $L_b$ lying above $\fp$,
	we have $r_{\fq}(r_i)=r_{\fq}(r_j)$ for some
	$1\leq i\neq j\leq d$.
\end{itemize} 
Then we have:
\begin{itemize}
	\item [(a)] $\displaystyle\sum_{\fp\in \cS_b}N_{\fp}\leq C_{24}h_K(b)+C_{25}$.
	\item [(b)] If $\fp\in M_K^0\setminus \cS_b$,
	$a\in \bP^1(K)$ that is not a pole of $f$ such that $\ord_{\fp}(f(a)-b)>0$,
	then there is $i\in\{1,\ldots,d\}$
	and a place $\fq$ of $L_b$ lying above $\fp$ such that
	$r_{\fq}(a)=r_{\fq}(r_i)$.
\end{itemize} 
\end{lemma}
\begin{proof}
Let $T$ be the set of prime $\fq$ of $L_b$ such that there
exist $1\leq i\neq j\leq d$ satisfying $r_{\fq}(r_i)=r_{\fq}(r_j)$. There exist constants $C_{26}$ and $C_{27}$
depending only on $K$ and $f$ such that:
$$\sum_{\fq\in T} N_{\fq}\leq C_{26}\max_{1\leq i\leq d}h_{L_b}(r_i)+C_{27}.$$
This implies part (a). 

For part (b), we have $\ord_{\fq}(f(a)-b)>0$ for some place
$\fq$ of $L$ lying above $\fp$. This implies
$r_{\fq}(f(a))=r_{\fq}(b)$. Since the $r_{\fq}(r_i)$'s
are distinct and the map $f$ mod $\fq$ has degree $d$, the elements
$r_{\fq}(r_i)$'s are exactly the preimages
of $r_{\fq}(b)$ under the map $f$ modulo $\fq$. Hence
there is some $r_i$ such that
$r_{\fq}(a)=r_{\fq}(r_i)$.
\end{proof}

We now spend the rest of this section to prove Theorem~\ref{thm:main}. Let $K$, $\varphi(x)$, $d$, $\tau$, $S$, and
$N_S$
be as in the statement of Theorem~\ref{thm:main}, write $\hhat=\hhat_{\varphi,K}$. We use the facts that 
$\hhat-h_K=O(1)$ and $\hhat\circ\varphi=d\cdot \hhat$
repeatedly.
To simplify
the exposition, the notations 
$C_{28},C_{29},\ldots$ denote positive constants
that always depend on 
$K$ and $\varphi$. For instance, the expression
$C_{28}:=C_{28}(A,b,\gamma)$ indicates
that $C_{28}$ depends on the quantities
$A$, $b$, $\gamma$, \emph{and} depends, as always, on $K$ and 
$\varphi$. When $K$ is a number field,
we assume Conjecture~\ref{conj:surface}.

Let $C_{29}:=C_{29}(\tau)$ be such that $d^{C_{29}-4}\tau$
is greater than the canonical height of any ramification point of $\varphi^4$.
For $\alpha\in \bP^1(K)$
with $\hhat(\alpha)\geq\tau$, for 
 every $m\geq C_{29}$, we have: 
 \begin{equation}\label{eq: varphi^-4}
 \text{$\varphi^4$ is unramified over $\varphi^m(\tau)$.}
 \end{equation}

Hence Theorem~\ref{thm:main} follows from the
following slightly more precise result:
\begin{thm}\label{thm:phi^4}
Let $K$, $\varphi(x)$, $d$, $\tau$, $S$, and $N_S$
be as in Theorem~\ref{thm:main}. If $K$ is a number field, assume Conjecture~\ref{conj:surface}. 
Then there exists a constant $C_{30}:=C_{30}(N_S,\tau)$
depending on $K$, $\varphi$, $N_S$, and $\tau$
such that the following holds. For every $\alpha\in\bP^1(K)$ such that $\hhat(\alpha)\geq \tau$,
for every $m\in \N_0$ such that
$\varphi^m(\alpha)\neq \infty$ and 
$\varphi^4$ is unramified over $\varphi^m(\alpha)$,
if $n>C_{30}$ and $\varphi^{m+n}(\alpha)\neq \infty$
then there exists $\fp\in M_K\setminus S$ such that
$\alpha$ has squarefree portrait $(m,n)$ modulo $\fp$.
\end{thm}

\begin{proof}
Applying Corollary~\ref{cor:double sequence} and Corollary~\ref{cor:double sequence ff} (for
$\ell=4$, $\epsilon=1$, and $a=b=\varphi^m(\alpha)$), there exist 
$C_{31}:=C_{31}(\tau)$, $C_{32}$, and $C_{33}$
such that for $n\geq C_{31}$, we have:
\begin{equation}\label{eq:main Zmn}
	\sum_{\fp\in \cD} N_{\fp}\geq (d^4-3)h_K(\varphi^{m+n-4}(\alpha))-C_{32}h_K(\varphi^m(\alpha))-C_{33}
\end{equation}
where $\cD:=\{\fp\in M_K^0:\ \ord_{\fp}(\varphi^{m+n}(\alpha)-\varphi^m(\alpha))\geq 1\}$. Let 
$\cD_{(1)}:=\{\fp\in M_K^0:\ \ord_{\fp}(\varphi^{m+n}(\alpha)-\varphi^m(\alpha))= 1\}$. After rewriting inequality \eqref{eq:main Zmn} and using $\vert \hhat-h\vert=O(1)$
and $\hhat\circ \varphi=d\hhat$, we have:
\begin{equation}\label{eq:cD1+cD}
\sum_{\fp\in \cD_{(1)}} N_{\fp}
+\sum_{\fp\in \cD\setminus\cD_{(1)}} N_{\fp}
\geq (d^4-3)d^{m+n-4}\tau-C_{32}d^m\tau-C_{34}.
\end{equation}
From the definition of $h_K$, we have:
\begin{equation}\label{eq:cD1+2cD}
h_K(\varphi^{m+n}(\alpha)-\varphi^m(\alpha))\geq \sum_{\fp\in \cD_{(1)}} N_{\fp}+2\left(\sum_{\fp\in \cD\setminus\cD_{(1)}} N_{\fp}\right)
\end{equation}
which implies:
\begin{equation}\label{eq:rewrite cD1+2cD}
d^{m+n}\tau+d^m\tau+C_{35}\geq \sum_{\fp\in \cD_{(1)}} N_{\fp}+2\left(\sum_{\fp\in \cD\setminus\cD_{(1)}} N_{\fp}\right).
\end{equation}
We combine inequalities \eqref{eq:cD1+cD} and \eqref{eq:rewrite cD1+2cD} to obtain:
\begin{align}\label{eq:cD1}
\begin{split}
\sum_{\fp\in \cD_{(1)}} N_{\fp}&\geq 2(d^4-3)d^{m+n-4}\tau-d^{m+n}\tau-C_{36}d^m\tau-C_{37}\\
&=(d^4-6)d^{m+n-4}\tau-C_{36}d^m\tau-C_{37}.
\end{split}
\end{align}

Let $\cE_1:=\{\fp\in M_K^0:\ r_{\fp}(\varphi^{m+n-1}(\alpha))=r_{\fp}(\varphi^{m-1}(\alpha))\}$. By
\eqref{eq:new Liouville}, we have:
\begin{align}\label{eq:cE1}
\begin{split}
\sum_{\fp\in\cE_1}N_{\fp}&\leq h_K(\varphi^{m-1}(\alpha))+h_K(\varphi^{m+n-1}(\alpha))+h_K(\varphi^{m-1}(\alpha))+C_{38}\\
&\leq d^{m+n-1}\tau+2d^{m-1}\tau+C_{39}
\end{split}
\end{align}

Let $\cE_2$ be the set of primes $\fp\in M_K^0$
of good reduction such that $r_{\fp}(\varphi^{m+n'}(\alpha))=r_{\fp}(\varphi^m(\alpha))$ for some $1\leq n'<n$ and $n'\mid n$. In other words, 
after reduction modulo $\fp$, $\varphi^m(\alpha)$
is periodic and its minimum period strictly divides $n$. 
By \eqref{eq:new Liouville}, we have:
\begin{equation}\label{eq:cE2 large m n}
\sum_{\fp\in\cE_2}N_{\fp}\leq \sum_{\text{prime } p\mid n}(2h_K(\varphi^{m}(\alpha))+
h_K(\varphi^{m+n/p}(\alpha))+C_{38}).
\end{equation}
Since there are at most $\log_2 n$ such $p$, we have:
\begin{equation}\label{eq:cE2 0}
\sum_{\fp\in\cE_2}N_{\fp}\leq (\log_2n) d^{m+n/2}\tau+2(\log_2n)d^m\tau+C_{40}\log_2 n\leq d^{m+2n/3}\tau
\end{equation}
when $n$ is larger than some constant $C_{41}(\tau)$.

Applying Lemma~\ref{lem:Sb} for $f=\varphi^4$ and $b=\varphi^m(\alpha)$, we get the resulting constants
$C_{42}$ and $C_{43}$ as in the conclusion
of Lemma~\ref{lem:Sb}. We also use the notation
$\cS_{\varphi^m(\alpha)}$ as in the statement of Lemma~\ref{lem:Sb}. 
Combining \eqref{eq:cD1}, \eqref{eq:cE1}, \eqref{eq:cE2 0},
and the observation that
$\displaystyle 1-\frac{6}{d^4}-\frac{1}{d}\geq \frac{1}{8}$,
we have:
\begin{align}\label{eq:cD1-cE1-cE2}
\begin{split}
\sum_{\fp\in\cD_{(1)}}N_{\fp}-\sum_{\fp\in\cE_1}N_{\fp}-\sum_{\fp\in\cE_2}N_{\fp}&\geq \frac{1}{8}d^{m+n}\tau-d^{m+2n/3}\tau-C_{44}d^m\tau-C_{45}\\
&>N_S+C_{42}h_K(\varphi^m(\alpha))+C_{43}
\end{split}
\end{align}
when $n>C_{46}(\tau,N_S)$. Hence there is a prime $\fp\in \cD_{(1)}$
such that $\fp\notin \cE_1\cup\cE_2\cup S\cup\cS_{\varphi^m(\alpha)}$. We have the following properties of $\fp$:
\begin{itemize}
\item [(i)] $\varphi$ has good reduction modulo $\fp$.
\item [(ii)] Let $L$ be the field obtained by adjoining 
the solutions of $\varphi^4(x)=\varphi^m(\alpha)$ to $K$. There
exist a prime $\fq$ of $L$ lying above $\fp$ and $r\in L$
with $\varphi^4(r)=\varphi^m(\alpha)$
and $r_{\fq}(r)=r_{\fq}(\varphi^{m+n-4}(\alpha))$. Consequently,
$r_{\fp}(\varphi^{m+n}(\alpha))=r_{\fp}(\varphi^m(\alpha))$.
\item [(iii)] $r_{\fp}(\varphi^{m+n-1}(\alpha))\neq r_{\fp}(\varphi^{m-1}(\alpha)))$ since $\fp\notin\cE_1$.
\item [(iv)] $r_{\fp}(\varphi^{m+n'}(\alpha))\neq r_{\fp}(\varphi^m(\alpha))$ for any proper divisor $n'$ of $n$.
\end{itemize}

Lemma~\ref{lem:easy} together with Properties (ii), (iii), (iv)
give that $\alpha$ has portrait $(m,n)$ modulo $\fp$. And since
$\fp\in \cD_{(1)}$, we finish the proof.
\end{proof}

\section{Proof of Theorem~\ref{thm:extra}}
\label{sec:proof extra}
We prove Theorem~\ref{thm:extra} by considering 3 cases:
\begin{itemize}
	\item [(a)] the case when both $m$ and $n$ are 
	sufficiently large
	which has been treated in Theorem~\ref{thm:main},
	\item [(b)] the case when $m$ is small (hence, we can fix $m$), and
	\item [(c)] the case when $n$ is small (hence, we can fix $n$). 
\end{itemize}

We will use the following simple result repeatedly
for the proof of Theorem~\ref{thm:extra}. It plays
the role of the condition \eqref{eq: varphi^-4}
in the proof of Theorem~\ref{thm:main}.
\begin{lemma}\label{lem:unramified backward}
Assume $\varphi$ is dynamically unramified over $t\in\bP^1(\bar{K})$. Then there
is a positive integer $i\leq 2d^3$ 
and $\gamma\in \bar{K}$ such that:
\begin{itemize}
	\item $\varphi^i(\gamma)=t$,
	\item $\varphi^i$ is unramified at $\gamma$, and
	\item $\varphi^3$ is unramified over $\gamma$.
\end{itemize}
\end{lemma}
\begin{proof}
	Since $\varphi$ is dynamically unramified over $t$,
	there is a backward orbit of distinct 
	elements
	$t_0=t,t_1,t_2,\ldots,t_{2d^3} \in \bP^1(\bar{K})$
	such that $\varphi(t_j)=t_{j-1}$
	and $\varphi$ is unramified at $t_j$ for $1\leq j\leq 2d^3$.
	Since $\varphi^3$ has at most $2d^3-2$ critical values, there are at least $2$ elements
	in $\{t_1,\ldots,t_{2d^3}\}$
	that are not critical value of $\varphi^3$. Let
	$t_i$ be such an element that is not $\infty$.
	This $t_i$ is our desired $\gamma$.
\end{proof}

\subsection{The case of Theorem~\ref{thm:extra} when $m$ is small}
We \emph{fix} $m$ such that $m\in A_1(\varphi,\alpha)$ and $\varphi^m(\alpha)\neq \infty$.
Then we prove that for all sufficiently large $n$, there is $\fp$ such that $\alpha$ has
squarefree portrait $(m,n)$ modulo $\fp$.
\begin{prop}\label{prop:small m}
In the number field case, assume Conjecture~\ref{conj:surface}. Let 
$\varphi$, $d$, $S$, $\hhat$, $\tau>0$, $\alpha\in\bP^1(K)$
be as in
Theorem~\ref{thm:extra}. Let $m\in A_1(\varphi,\alpha)$ such that $\varphi^m(\alpha)\neq \infty$. We have
the following:
\begin{itemize}
	\item [(a)] If $m=0$, there exists $C_{47}(\tau,N_S)$ such that for
	every
	$n>C_{47}(\tau,N_S)$,
	there is a prime $\fp\notin S$ such that
	$\alpha$ is periodic modulo $\fp$ with exact period $n$ and 
	$\ord_{\fp}(\varphi^n(\alpha)-\alpha)=1$ (i.e. $\alpha$ has squarefree portrait $(0,n)$ modulo $\fp$).
	
	\item [(b)] If $m\geq 1$, there exists $C_{48}(\tau,N_S,m)$ such that for
	every
	$n>C_{48}(\tau,N_S,m)$,
	there is a prime $\fp\notin S$ such that
	$\alpha$ has squarefree portrait $(m,n)$ 
	modulo $\fp$.
\end{itemize}
\end{prop}
\begin{proof}     
	If $\varphi^4$ is unramified over $\varphi^m(\alpha)$, Theorem~\ref{thm:phi^4} gives the desired
    conclusion. Let $\{b_1,\ldots,b_k\}$
    denote the set of critical values
    of $\varphi^4$ in $K$; this set depends only on
    $\varphi$ and $K$. Assume this set is not empty, otherwise we are done.
    It remains to treat the case
    $\varphi^m(\alpha)\in\{b_1,\ldots,b_k\}$.
    
    Let $n\in\N$ such that
    $\varphi^{m+n}(\alpha)\neq \infty$.
    If $m\geq 1$, there is $\eta\in \bP^1(\bar{K})$
    such that 
    $\eta\neq \varphi^{m-1}(\alpha)$, 
    $\varphi(\eta)=\varphi^m(\alpha)=b_1$, $\varphi$ is unramified at $\eta$ and dynamically unramified
    over $\eta$. 
    By Lemma~\ref{lem:unramified backward}, there
    exist an integer $i$
    such that $1\leq i-1\leq 2d^3$
    and $\gamma\in \bar{K}$ satisfying the following:
    $\varphi^{i-1}(\gamma)=\eta$, 
    $\varphi^{i-1}$ is unramified at $\gamma$, and 
    $\varphi^3$ is unramified over $\gamma$. This
    last condition means the function
    $\varphi^3(x)-\gamma$ has $d^3$ distinct zeros.
   
	Similarly, if $m=0$, by Lemma~\ref{lem:unramified backward} and
	the fact that $\varphi$ is dynamically unramified over 
	$\alpha$, there exists a positive integer $i$ such that $i\leq 2d^3$ and an element $\gamma\in\bar{K}$
	satisfying the following: $\varphi^i(\gamma)=\alpha$,
	$\varphi^i$ is unramified at
	$\gamma$, and $\varphi^3$ is unramified over 
	$\gamma$.
	
	In both cases ($m=0$ and $m\geq 1$), we have that
	$i\leq 2d^3+1$. We recall that $\varphi^m(\alpha)$
	belongs to the list $\{b_1,\ldots,b_k\}$
	which depends only on $\varphi$ and $K$. Hence
	when $m\geq 1$ (respectively $m=0$)
	the triple $(i,\eta,\gamma)$
	(respectively, the pair $(i,\gamma)$)
	belongs to a finite set that depends only on $K$
	and $\varphi$. Hence we may replace 
	$K$ by $K(\gamma)$
	since every constant that depends on $\varphi$, 
	$K(\gamma)$, and ``other data'' will ultimately depend
	on $\varphi$, $K$, and the exact same ``other data''.
	In other words, we may assume $\gamma\in K$.
	Then we can write:
	$$\varphi^i(x)-\varphi^m(\alpha)=\frac{(x-\gamma)F(x)}{G(x)}$$
	where $F(x),G(x)\in K[x]$ with
	$F(\gamma)G(\gamma)\neq 0$ and $\gcd(F(x),G(x))=1$.
	We also require that $G(x)$ is monic; hence
	the pair $(F,G)$ is determined uniquely.

	Write $F(x)=c\prod_{j=1}^{r} (x-f_j)$
	and $G(x)=\prod_{k=1}^{s} (x-g_k)$. Let $S'\subset M_K^0$ be the finite set of primes $\fp$ satisfying
	one of the following conditions:
	\begin{itemize}
		\item $\vert c\vert_\fp\neq 1$
		\item $\vert \gamma-f_j\vert_\fp\neq 1$ for some $j$.
		\item $\vert \gamma -g_k\vert_\fp\neq 1$ for some $k$.		
	\end{itemize}
	The point is that for every $z\in K$ that
	is not a zero or pole
	of $\varphi^i(x)-\varphi^m(\alpha)$
	and for every $\fp\in M_K^0\setminus S'$,
	if $\ord_{\fp}(z-\gamma)>0$ then
	$\ord_{\fp}(\varphi^i(z)-\varphi^m(\alpha))=\ord_{\fp}(z-\gamma)$. As argued before, the list of
	possibilities for $S'$ depends only on $\varphi$
	and $K$. Hence there is $C_49$ such that:
	\begin{equation}\label{eq:C56 C57}
	\sum_{\fp\in S'}N_{\fp}\leq C_{49}
	\end{equation}
	
	Applying Corollary~\ref{cor:double sequence}
	and Corollary~\ref{cor:double sequence ff}
	(for $\ell=3$, $\epsilon=1$,
	$b=\gamma$, and $a=\varphi^{m+n-i-3}(\alpha)$),
	there exist positive constants 
	$C_{50}$, $C_{51}$, and $C_{52}$
	such that when $m+n-i-3\geq C_{50}$, we have
	$\varphi^{m+n-i-3}(\alpha)\neq\infty$
	and the set
	$\cD:=\{\fp\in M_K^0:\ \ord_{\fp}(\varphi^{m+n-i}(\alpha)-\gamma)\geq 1\}$
	satisfies:
	\begin{equation}\label{eq:C58-60}
	\sum_{\fp\in \cD} N_{\fp}\geq (d^3-3)h_K(\varphi^{m+n-i-3}(\alpha))-C_{51}h_K(\gamma)-C_{52}.
	\end{equation}
	Let $\cD_{(1)}:=\{\fp\in M_K^0:\ \ord_{\fp}(\varphi^{m+n-i}(\alpha)-\gamma)=1\}$. After rewriting
	inequality \eqref{eq:C58-60} and using
	$\hhat(\gamma)=\displaystyle\frac{\hhat(\alpha)}{d^i} $, we have:
	\begin{equation}\label{eq:C61-62}
	\sum_{\fp\in \cD_{(1)}}N_{\fp}+\sum_{\fp\in \cD\setminus\cD_{(1)}} N_{\fp}\geq (d^3-3)d^{m+n-i-3}\tau-C_{53}\tau-C_{54}.	
	\end{equation}
	
	On the other hand, we have:
	\begin{equation}\label{eq:C63}
	\sum_{\fp\in \cD_{(1)}}N_{\fp}+2\sum_{\fp\in \cD\setminus\cD_{(1)}} N_{\fp}\leq h_K(\varphi^{m+n-i}(\alpha)-\gamma)\leq d^{m+n-i}\tau+\tau+C_{55}
	\end{equation}
	
	Inequalities \eqref{eq:C61-62} and \eqref{eq:C63}
	imply:
	\begin{equation}\label{eq:C64-65}
	\sum_{\fp\in \cD_{(1)}}N_{\fp}\geq (d^3-6)d^{m+n-i-3}\tau-C_{56}\tau-C_{57}.
	\end{equation}
	
	Let $\cE$ be the set of primes $\fp\in M_K^0$
	of good reduction
	such that $r_{\fp}(\varphi^m(\alpha))=r_{\fp}(\varphi^{m+n'}(\alpha))$
	for some $1\leq n'<n$ with $n'\mid n$.
	By the same arguments giving
	\eqref{eq:cE2 large m n} and \eqref{eq:cE2 0},
	we have:
	\begin{align}\label{eq:C66}
\begin{split}
\sum_{\fp\in\cE}N_{\fp}&\leq \sum_{\text{prime }p\mid n}(2h_K(\varphi^m(\alpha))+h_K(\varphi^{m+\frac{n}{p}}(\alpha))+C_{58})\\
&\leq (2d^m\tau+d^{m+\frac{n}{2}}\tau+C_{59})\log_2n
\leq d^{m+\frac{2n}{3}}\tau
\end{split}
\end{align}
	when $n$ is larger than some constant $C_{60}(\tau)$.
	Since $i\leq 2d^3+1$, the term
	$(d^3-6)d^{m+n-i-3}$
	dominates the term $d^{m+\frac{2n}{3}}$ when
	$n$ is sufficiently large.
	
	If $m=0$, from \eqref{eq:C56 C57},\eqref{eq:C64-65},
	and \eqref{eq:C66}, there exists
	$C_{61}(\tau,N_S)$ such that
	$$\sum_{\fp\in\cD_{(1)}}N_{\fp}
	-\sum_{\fp\in\cE}N_{\fp}-
	\sum_{\fp\in S'}N_{\fp}-N_S>0$$
	when $n>C_{61}(\tau,N_S)$. Consequently, there
	is $\fp\in \cD_{(1)}\setminus(\cE\cup S'\cup S)$.
	Since $\fp\in\cD_{(1)}\setminus S'$,
	we have:
	$$\ord_{\fp}(\varphi^{n}(\alpha)-\alpha)=\ord_{\fp}(\varphi^{n-i}(\alpha)-\gamma)=1.$$
	Since $\fp\notin S'$, we have that when
	reducing modulo $\fp$, the period 
	of $\alpha$ must be exactly $n$.
	This proves part (a).
	
	If $m\geq 1$, let $\cE_1:=\{\fp\in M_K^0:\ r_{\fp}(\eta)=r_{\fp}(\varphi^{m-1}(\alpha))\}$. Since
	$\eta\neq \varphi^{m-1}(\alpha)$ and $\varphi(\eta)=\varphi^m(\alpha)$, inequality \eqref{eq:new Liouville}
	gives:
	\begin{equation}\label{eq:C69}
	\sum_{\fp\in \cE_1}N_\fp\leq 3d^{m-1}\tau+C_{62}.
	\end{equation}
	From \eqref{eq:C56 C57},\eqref{eq:C64-65},\eqref{eq:C66}, and \eqref{eq:C69}, there exists
	$C_{63}(\tau,N_S,m)$ such that
	$$\sum_{\fp\in\cD_{(1)}}N_{\fp}
	-\sum_{\fp\in\cE}N_{\fp}-
	\sum_{\fp\in \cE_1}N_{\fp}-\sum_{\fp\in S'}N_\fp-N_S>0$$
	when $n>C_{63}(\tau,N_S,m)$. Consequently, there
	is $\fp\in \cD_{(1)}\setminus(\cE\cup\cE_1\cup S'\cup S)$. By similar arguments in the case $m=0$, we
	have $\ord_{\fp}(\varphi^{m+n}(\alpha)-\varphi^m(\alpha))=1$ and $\varphi^m(\alpha)$
	is periodic with exact period $n$ modulo $\fp$. It remains to
	show that $\varphi^{m-1}(\alpha)$ is not periodic
	modulo $\fp$. Since $\fp\notin\cE_1$ and $r_{\fp}(\varphi^{m+n-i}(\alpha))=r_{\fp}(\gamma)$, we have $r_{\fp}(\varphi^{m+n-1}(\alpha))=r_{\fp}(\eta)\neq r_{\fp}(\varphi^{m-1}(\alpha))$. This finishes the proof.
	\end{proof}
	
As an application of Proposition~\ref{prop:small m},
we can now prove that the set $\N\setminus A_2(\varphi)$
is finite:
\begin{prop}\label{prop:cofinite A1A2}
Let $F$ be a field of characteristic $0$, let $\varphi(x)\in F(x)$ with $d:=\deg(\varphi)\geq 2$,
and let $\alpha\in \bP^1(F)$ that is not $\varphi$-preperiodic. Then the sets $\N_0\setminus A_1(\varphi,\alpha)$ and
$\N\setminus A_2(\varphi)$ are finite. 
\end{prop}
\begin{proof}
Since replacing $F$ by $\bar{F}$ does not change 
$A_1(\varphi,\alpha)$ and $A_2(\varphi)$, we may
assume $F=\bar{F}$. 

It is easy to show that $\N\setminus A_1(\varphi,\alpha)$
is finite, as follows.
Let $c_1,\ldots,c_k$ be 
the distinct critical values of $\varphi$; we have
$k\leq 2d-2$. Let $\cO_i:=\{\varphi^n(c_i):\ n\geq 0\}$
for $1\leq i\leq k$. If $\varphi^m(\alpha)$ is
not a critical value of $\varphi^3$
then the set 
$$A:=(\varphi^3)^{-1}(\varphi^m(\alpha))\setminus
(\varphi^2)^{-1}(\varphi^{m-1}(\alpha))$$
has $d^3-d^2$ elements. Since $\alpha$ is not
preperiodic, $\cO_i\cap A$ has at most one element
for each $1\leq i\leq k$. And since $k\leq 2d-2<d^3-d^2$,
there is $\gamma\in A$ that does not belong to any $\cO_i$. By choosing $\eta=\varphi^2(\gamma)\neq \varphi^{m-1}(\alpha)$, we have verified
that $m\in A_1(\varphi,\alpha)$.

Now we prove that $\N\setminus A_2(\varphi)$ is finite.
We consider the function field $K=F(t)$, the isotrivial
function $\varphi(t)\in K(t)$, and the 
starting point $\alpha=t\in K$. By a direct computation, we have $\hhat_{\varphi,K}(\alpha)=1$. Let $\tau=1$
and let $S$ be the singleton whose element is 
the place at infinity of $K=F(t)$
(i.e. the place corresponding the point $\infty$ in
$\bP^1(F)$). Fix $m=0$, then Proposition~\ref{prop:small m} gives that for all sufficiently large $n$,
there is $\fp\in M_K\setminus S$
such that $\alpha$ has squarefree portrait $(m,n)$
modulo $\fp$. In other words, if $\beta\in F$ is the point on $\bP^1(F)$
corresponding $\fp$ then we have:
\begin{itemize}
	\item $\beta$ is period of exact period $n$ under
	the function
	$\varphi$, and 
	\item $t-\beta$ is a squarefree factor of $\varphi^n(t)-t$.
\end{itemize}
As above, $\cO_i$ for $1\leq i\leq k$ denotes the critical orbits. Among all (possibly none) of the $\cO_i$'s that are finite
(equivalently, $c_i$ is preperiodic), let $N$
be the maximum of the sizes of the periodic cycles.
If we require further that $n>N$ then $\beta$
is not contained in any $\cO_i$. This implies
$n\in A_2(\varphi)$. 
\end{proof}
	
\subsection{The case of Theorem~\ref{thm:extra} when
$n$ is small}
Fix $n\in A_2(\varphi)$. Then we prove
that for all sufficiently large $m$,
there is $\fp\in M_K\setminus S$ such that $\alpha$ has squarefree portrait
$(m,n)$ modulo $\fp$.
\begin{prop}\label{prop:small n}
In the number field case, assume Conjecture~\ref{conj:surface}. Let $\varphi$, $d$, $S$, 
$\hhat$, $\tau$, and $\alpha\in\bP^1(K)$ be as
in Theorem~\ref{thm:extra}.
Let $n\in A_2(\varphi)$.
There exist a positive constant
$C_{64}(\tau,N_S,n)$
such that the following holds. For every
$m>C_{64}(\tau,N_S,n)$, if $\infty\notin\{\varphi^m(\alpha),\varphi^{m+n}(\alpha)\}$
then there is $\fp\in M_K\setminus S$ such that 
$\alpha$ has squarefree portrait $(m,n)$ modulo $\fp$. 
\end{prop}
\begin{proof}
	The proof uses similar arguments in the
	proof of Proposition~\ref{prop:small m} so we will be 
	brief.
	Let $\beta$ and $\eta$ be as in 
	Definition~\ref{def:A1A2A}(ii). 
	By Lemma~\ref{lem:unramified backward},
	there is an integer $i$ with $1\leq i-1\leq 2d^3$
	and $\gamma\in\bar{K}$ such that
	$\varphi^{i-1}(\gamma)=\eta$, $\varphi^{i-1}$ is 
	unramified
	at $\gamma$ and 
	$\varphi^3$ is unramified over $\gamma$.
	As in the proof of Proposition~\ref{prop:small m},
	since the data $(\beta,\eta,i,\gamma)$
	belongs to a finite set depending only on 
	$\varphi$, $n$, and $K$,
	after extending $K$ if necessary, 
	we may assume that $\gamma\in K$.

	Now $x-\gamma$ is a squarefree factor
	of $\varphi^i(x)-\beta$. Since $x-\beta$
	is a squarefree factor of $\varphi^n(x)-x$, we have:
	$$\varphi^{n+i}(x)-\varphi^{i}(x)=\frac{(x-\gamma)F(x)}{G(x)}$$
	with $F(x),G(x)\in K[x]$, $F(\gamma)G(\gamma)\neq 0$,
	and $\gcd(F,G)=1$. As before, choose $G$ to be monic
	so that the pair $(F,G)$ is uniquely determined.
	
	As in the proof of Proposition~\ref{prop:small m},
	we have a finite subset $S'$ of $M_K^0$
	depending on $K$, $\varphi$, $n$, and $i$
	such that 
	for every $\fp\in M_K^0\setminus S'$
	and every $z\in K$ that is not a zero
	of $(x-\gamma)F(x)G(x)$,
	if $\ord_{\fp}(z-\gamma)>0$
	then $\ord_{\fp}(\varphi^{n+i}(z)-\varphi^i(z))= \ord_{\fp}(z-\gamma)$.
	
	Let $m\in \N_0$ such that
	$m\geq i+3$ and $\infty\neq \varphi^{m-i}(\alpha)$.
	As in the proof of Proposition~\ref{prop:small m},
	we can define the following sets:
	$$\cD_{(1)}:=\{\fp\in M_K^0:\ \ord_{\fp}(\varphi^{m-i}(\alpha)-\gamma)=1\},$$
	$$\cE:=\{\fp\in M_K^0:\ r_{\fp}(\beta)=r_{\fp}(\varphi^{n'}(\beta))\ \text{for some $1\leq n'<n$}\},$$
	$$\cE_1:=\{\fp\in M_K^0:\ r_{\fp}(\eta)=r_{\fp}(\varphi^{n-1}(\beta))\}.$$
	
	Then we can prove that when $m$ is sufficiently large, there is $\fp\in\cD_{(1)}\setminus (\cE\cup\cE_1\cup S'\cup S)$. This gives:
	\begin{itemize}
	\item $\ord_{\fp}(\varphi^{m+n}(\alpha)-\varphi^m(\alpha))=1$ since $\fp\in\cD_{(1)}\setminus S'$.
	\item $r_{\fp}(\varphi^m(\alpha))=r_{\fp}(\varphi^i(\gamma))=r_{\fp}(\beta)$
	which has exact period $n$ since $\fp\notin\cE$.	
	\item $r_{\fp}(\varphi^{m-1}(\alpha))=r_{\fp}(\varphi^{i-1}(\gamma))=r_{\fp}(\eta)\neq r_{\fp}(\varphi^{n-1}(\beta))$, hence
	$\varphi^{m-1}(\alpha)$ is not periodic modulo $\fp$.
	\end{itemize}
	This finishes the proof.
	\end{proof}






\begin{thebibliography}{GNT2015}
\newcommand{\au}[1]{{#1},}
\newcommand{\ti}[1]{\textit{#1},}
\newcommand{\jo}[1]{{#1}}
\newcommand{\vo}[1]{\textbf{#1}}
\newcommand{\yr}[1]{(#1),}
\newcommand{\no}[1]{(#1),}
\newcommand{\pp}[1]{#1.}
\newcommand{\ppx}[1]{#1,}
\newcommand{\pps}[1]{#1;}
\newcommand{\bk}[1]{{#1},}
\newcommand{\inbk}[1]{in: {#1}}
\newcommand{\xxx}[1]{{arXiv:#1.}}

\bibitem[Bak09]{Baker2009}
M.~Baker, \emph{A finiteness theorem for canonical heights attached to rational maps over function fields}, 
J. reine angew. Math \textbf{626} (2009), 205--233.

\bibitem[Ban86]{Bang}
A.~S.~Bang, \emph{Taltheoretiske Undersogelse}, Tidsshrift Mat. \textbf{4} (1886) 70--80, 130--137.

\bibitem[BG06]{bg06}
\au{E.~Bombieri and W.~Gubler}
\ti{Heights in Diophantine Geometry}
\jo{New Mathematical Monograph, Cambridge Univ. Press, Cambridge}
\vo{3}
\yr{2006}

\bibitem[CS93]{Call-Silverman}
\au{G.~S.~Call and J.~H.~Silverman}
\ti{Canonical heights on varieties with morphisms}
\jo{Compositio Math.}
\vo{89}
\yr{1993}
\pp{163--205}

\bibitem[DH12]{DH2012}
K.~Doerksen and A.~Haensch, \emph{Primitive prime divisors in zero orbits of polynomials}, Integers
\textbf{12} (2012), 7pp.

\bibitem[Doy]{DoylePreprint}
J.~Doyle, \emph{Preperiodic portraits for unicritical polynomials over a rational function field}, preprint, arXiv:1603.08138.

\bibitem[Doy16]{Doyle2016}
J.~Doyle, \emph{Preperiodic portraits for unicritical
polynomials}, Proc. Amer. Math. Soc. \textbf{144} (2016),
2885--2899.


\bibitem[EMW06]{EMW}
G.~Everest, G.~Mclaren, and T.~Ward, \emph{Primitive divisors
of elliptic divisibility sequences}, J. Number Theory \textbf{118} (2006), 71--89.

\bibitem[FG11]{FG}
\au{X.~Faber and A.~Granville}
\ti{Prime factors of dynamical sequences}
\jo{J. Reine Angew. Math.}
\vo{661}
\yr{2011}
\pp{189--214}

\bibitem[GNT15]{portrait} 
D.~Ghioca, K.~Nguyen, and T.~J.~Tucker, \emph{Portraits of preperiodic points for rational maps},
Math. Proc. Cambridge Philos. Soc. \textbf{159} (2015), 165--186.


\bibitem[GNT13]{GNT2013}
\au{C.~Gratton, K.~Nguyen, and T.~J.~Tucker}
\ti{$ABC$ implies primitive prime divisors in arithmetic dynamics}
\jo{Bull. London Math. Soc.}
\vo{45}
\yr{2013}
\pp{1194--1208}



\bibitem[Ing07]{Ingram2007}
P.~Ingram, \emph{Elliptic divisibility sequences over certain curves}, J. Number Theory \textbf{123} (2007), 473--486.

\bibitem[IS09]{IS}
\au{P.~Ingram and J.~H.~Silverman}
\ti{Primitive divisors in arithmetic dynamics}
\jo{Math. Proc. Cambridge Philos. Soc.}
\vo{146}
\yr{2009}
\pp{289--302}


\bibitem[Kri13]{Krieger}
\au{H.~Krieger}
\ti{Primitive prime divisors in the critical orbit of $z^d+c$} Int. Math. Res. Not. \emph{IMRN} (2013), No. 23, 5498--5525.

\bibitem[Sch74]{Schinzel}
A.~Schinzel, \emph{Primitive divisors of the expression $a^n-b^n$ in algebraic number fields}, J. Reine Angew. Math. \textbf{268/269} (1974), 27--33.

\bibitem[Sil07]{Silverman-book}
J.~H.~Silverman, \emph{The arithmetic of dynamical systems}, Graduate Texts in Mathematics {\bf 241}, Springer, New York, 2007, 511 pp.

\bibitem[Sil13]{Silverman2013}
J.~H.~Silverman, \emph{Primitive divisors, dynamical Zsigmondy sets, and Vojta's conjecture}, J. Number Theory \textbf{133} (2013), 2948--2963.

\bibitem[Voj11]{Vojta2011} 
P.~Vojta, \emph{Diophantine approximation and Nevanlinna theory}, Arithmetic Geometry.
Lectures given at the C.I.M.E. summer school held in Cetraro, Italy, September 10--15,
2007, Lecture Notes in Math., no. 2009, Springer-Verlag, 2011, pp. 111--230.

\bibitem[Yam04]{Yamanoi-2004} 
K.~Yamanoi, \emph{The second main theorem for small functions and related problems}, Acta Math. \textbf{192} (2004), 225--294.

\bibitem[Zsi92]{Zsigmondy}
K.~Zsigmondy, \emph{Zur theorie der Potenzreste}, Monatsh. Math. Phys. \textbf{3} (1892), 265--284.

\end{thebibliography}
\end{document}